\documentclass[11pt,openany,leqno]{article}
\usepackage{amsmath,amsthm,amsfonts,amssymb,amscd,url}
\usepackage{graphics}
\usepackage[latin1]{inputenc}

\usepackage{enumerate}
\begin{document}
\def\RA{{\bf\text{To be read again}}}
\def\ERA{{\bf\text{End To be read again}}}
\def\Diff{\text{Diff}}
\def\Max{\text{max}}
\def\Log{\text{log}}
\def\loc{\text{loc}}
\def\inta{\text{int }}
\def\det{\text{det}}
\def\exp{\text{exp}}
\def\Re{\text{Re}}
\def\lip{\text{Lip}}
\def\leb{\text{Leb}}
\def\dom{\text{Dom}}
\def\diam{\text{diam}\:}
\newcommand{\ovfork}{{\overline{\pitchfork}}}
\newcommand{\ovforki}{{\overline{\pitchfork}_{I}}}
\newcommand{\Tfork}{{\cap\!\!\!\!^\mathrm{T}}}
\newcommand{\whforki}{{\widehat{\pitchfork}_{I}}}
\theoremstyle{plain}
\newtheorem{theo}{\bf Theorem}[section]
\newtheorem{lemm}[theo]{\bf Lemma}
\newtheorem{prop}[theo]{\bf Proposition}
\newtheorem{coro}[theo]{\bf Corollary}
\newtheorem{Property}[theo]{\bf Property}

\theoremstyle{remark}
\newtheorem{rema}[theo]{\bf Remark}
\newtheorem{exem}[theo]{\bf Example}
\newtheorem{Examples}[theo]{\bf Examples}
\newtheorem{defi}[theo]{\bf Definition}
\newtheorem{ques}[theo]{\bf Question}

\title{Persistence of laminations}

\author{Pierre Berger\\
Institute for Mathematical Sciences,\\
 Stony Brook University,\\
  Stony Brook NY 11794-3660,\\
   USA\\
\texttt{berger@phare.normalesup.org}}
\date{\today}

\maketitle
\begin{abstract}
We present a modern proof of some extensions of the celebrated Hirsch-Pugh-Shub theorem on persistence of normally hyperbolic compact laminations. Our extensions consist of allowing the dynamics to be an endomorphism, of considering the complex analytic case and of allowing the laminations to be non compact. To study the analytic case, we use the formalism of deformations of complex structures. We present various persistent complex laminations which appear in dynamics of several complex variables: H\'enon maps, fibered holomorphic maps...

In order to proof the persistence theorems, we construct a laminar structure on the stable and unstable of the normally hyperbolic laminations.\end{abstract}
{\footnotesize
 \tableofcontents}

\section*{Introduction}
In 1977, M. Hirsch, C. Pugh and M. Shub \cite{HPS} developed a theory which has been very useful for hyperbolic dynamical systems.
The central point of their work was to prove the $C^r$-persistence of manifolds, foliations, or more generally laminations which are $r$-normally hyperbolic and plaque-expansive, for all $r\ge 1$.

  We recall that a lamination is \emph{$r$-normally hyperbolic}, if the dynamics preserves the lamination (each leaf is sent into a leaf) and if the normal space of the leaves splits into two $Tf$-invariant subspaces, that $Tf$ contracts (or expands) $r$-times more sharply than the tangent space to the leaves. Hence, $0$-normal expansion (resp. contraction) means that the action of the dynamics on the normal bundle is expanding (resp. contracting), and so does not require the existence of a dominated splitting. \emph{Plaque-expansiveness} is a generalization of expansiveness to the concept of laminations. The \emph{$C^r$-persistence} of such a lamination means that for any $C^r$-perturbation of the dynamics, there exists a lamination, $C^r$-close to the first, which is  preserved by the new dynamics, and such that the dynamics induced on the space of the leaves remains the same.

    A direct application of this theory was the construction of an example of a robustly transitive  {diffeomorphism (every close diffeomorphism has a dense orbit) but not Anosov}.
Then their work was used for example by C. Robinson \cite{Rs} to achieve the proof of the Palis-Smale conjecture:   the $C^1$-diffeomorphisms that satisfy Axiom A and the strong transversality condition are structurally stable.

Nowadays, this theory remains very useful in several mathematical fields such  as generic dynamical systems, differentiable dynamics, foliations theory or Lie group theory.   

A first result is an analogue of the HPS theorem to the endomorphism case:

We allow $f$ to be an endomorphism that is to be possibly non-bijective and with singularities, but we suppose $f$ to be normally expanding instead of normally hyperbolic:

\begin{theo}\label{preth1}
Let $r\ge 1$ and let $(L,\mathcal L)$ be a compact lamination $C^r$-immersed by $i$ into a manifold $M$. Let $f$ be a $C^r$-endomorphism of $M$ which preserves and $r$-normally expands $\mathcal L$.  Then  the immersed lamination is $C^r$-persistent. 

 If moreover $i$ is an embedding and $f$ is forward plaque-expansive at $\mathcal (L,\mathcal L)$ then the embedded lamination is $C^r$-persistent.\end{theo}

 A second result is a generalization of the HPS theorem:

\begin{theo}\label{preth2}
Let $r\ge 1$ and let $(L,\mathcal L)$ be a compact lamination $C^r$-immersed by $i$ into a manifold $M$. Let $f$ be a $C^r$-endomorphism of $M$ which preserves and is $r$-normally hyperbolic to $\mathcal L$ with a bijective pullback $f^*$.  Then  the immersed lamination is $C^r$-persistent.

 If moreover $i$ is an embedding and $f^*$ is plaque-expansive at $\mathcal (L,\mathcal L)$ then the embedded lamination is $C^r$-persistent.\end{theo}
We recall that a \emph{pullback} of an endomorphism $f$ preserving a lamination $(L,\mathcal L)$ immersed by $i$ is an endomorphism $f^*$ of $(L,\mathcal L)$ such that:
\[i\circ f^*= f\circ i.\]

In the complex analytic world, we prove the following theorem:

\begin{theo}\label{defo-exp} Let $(L,\mathcal L)$ be a complex lamination immersed by i into a complex manifold $M$. Let $f$ be a holomorphic endomorphism of $M$ which preserves and $0$-normally expands $\mathcal L$.

Then for any complex $k$-ball $B$, and any complex analytic family $(f_t)_{t\in B}$ of endomorphisms of $M$, such that $f_{0}$ is $f$, there exist an open neighborhood $B_0$ of $0\in B$ and a complex analytic family of laminations $(L, \mathcal L_t)_{t\in B_0}$ immersed by $(i_t)_t$ such that:
\begin{itemize}
\item[-] $f_0=f$, $i_0=i$ and $\mathcal L_0=\mathcal L$,  
\item[-] $f_t$ preserves the lamination $(L,\mathcal L_t)$ holomorphically immersed by $i_t$ into $M$, for any $t\in B_0$.\end{itemize}

If moreover $i$ is an embedding and $f$ is forward plaque-expansive, then $i_t$ is an embedding for every $t\in B_0$.
\end{theo}

A \emph{complex  analytic family of endomorphisms} $(f_t)_{t\in B}$ of $M$ means that $(t,x)\in B\times M\mapsto (t,f_t(x))$ is complex analytic. 

By complex analytic family of laminations $(L,\mathcal L_t)_{t\in B_0}$ immersed by $(i_t)_t$, we mean that there exists a complex lamination structure $\mathcal D$ on $D:= B\times L$ such that:\begin{itemize}
\item[-] $\overline{w}\;: \; (x,t)\in (D,\mathcal D)\mapsto t\in B_0$ is a complex analytic submersion and  $\overline{w}^{-1}(t)= (L,\mathcal L_t)$,
\item[-] $(x,t)\in (D,\mathcal D)\mapsto (i_t(x),t)\in M\times B$ is a complex analytic immersion.
\end{itemize}

\begin{rema} In Example \ref{hopf}, we show that the complex structure $\mathcal D$ on $L\times B_0$ is not necessarily the trivial product structure $\mathcal L\times B_0$. Nevertheless, we give a sufficient condition to have such a product structure (see Proposition \ref{J3}).\end{rema}

\begin{rema} The above theorem remains true if $L$ is not necessarily compact but a pullback $f^*$ of $f$ sends $cl(L)$ into $L$.\end{rema} 

The mirror result of the generalization of HPS's theorem is the following:
\begin{theo}\label{defo-hyp} Let $(L,\mathcal L)$ be a compact complex lamination immersed by $i$ into a complex manifold $M$.
Let $f$ be a holomorphic endomorphism of $M$ which preserves, has bijective pull back $f^*$ and is $1$-normally hyperbolic to $(L,\mathcal L)$. Then for any complex $k$-ball $B$ and complex analytic family $(f_t)_{t\in B}$ of endomorphisms of $M$ such that $f_{0}$ is $f$, there exist an open neighborhood $B_0$ of $0$ and a complex analytic family of laminations $(L, \mathcal L_t)_{t\in B_0}$ immersed by $(i_t)_t$ such that:
\begin{itemize}
\item[-] $f_0=f$, $i_0=i$ and $\mathcal L_0=\mathcal L$,  
\item[-] $f_t$ preserves the lamination $(L,\mathcal L_t)$ holomorphically immersed by $i_t$ into $M$, for any $t\in B_0$.\end{itemize}

If moreover $i$ is an embedding and $f^*$ is plaque-expansive, then $i_t$ is an embedding for every $t\in B_0$.
\end{theo}

\begin{rema} Under the hypothesis of the above theorem, we have the same conclusion if $f$ is $0$-normally \emph{contracting} instead of $1$-normally hyperbolic.\end{rema}

We will provide some extensions of all the above results in the \emph{non-compact cases} (see Theorems \ref{Main} and \ref{Maincontract}).\\ 

We notice that in all the above theorems, the hypotheses are open.

Let us finish this introduction by giving two examples.
\begin{exem} 
Let $R_1$ and $R_2$ be two rational functions of the Riemannian sphere $\mathbb S^2$. Let $K$ be a compact subset of $\mathbb S^2$ expanded by $R_1$: $R_1$ sends $K$ into itself and there exists $\lambda>1$ such that for a Riemannian metric on $\mathbb S^2$, 
for every $x\in K$ $|d_xR_1|\ge \lambda$.
We suppose that for every $z\in \mathbb S^2$, $|d_z R_2|<\lambda$.

Let $f:\; (z,z')\in \mathbb S^2\times \mathbb S^2\rightarrow (R_1(z), R_2(z'))$. Let $\mathcal L$ be the lamination on $K\times \mathbb S^2$ whose leaves are of the form $\mathcal L_k:=\{k\}\times \mathbb S^2$, for $k\in K$. We notice that this compact lamination is preserved by $f$ and $1$-normally expanded by $f$. By Theorem \ref{preth1}, this lamination is $C^1$-persistent. In other words, for any $C^1$-perturbation $f'$ of $f$, there exists a family of  manifolds $(\mathcal L'_k)_{k\in K}$ such that for any $k\in K$:
\begin{itemize}
\item[-] $\mathcal L'_k$ is $C^1$-close to $\mathcal L_k$: there exists a $C^1$-embedding of $\mathcal L_k$ onto $\mathcal L'_k$ which is $C^1$-close to the canonical inclusion,  
\item[-] the endomorphism $f'$ sends $\mathcal L'_k$ into $\mathcal L'_{R_1(k)}$, 
\item[-] for $k'$ close to $k$ the sphere $\mathcal L'_{k'}$ is $C^1$-close to $\mathcal L_k$.  
\end{itemize}

Moreover, since $f$ is plaque-expansive at $\mathcal L$, the submanifolds $(\mathcal L'_k)_k$ are disjoint from each other. This implies that there exists a $C^1$-embedding, close to the canonical inclusion, of $\mathcal L$ onto the laminations $\mathcal L'$ formed by the leaves $(\mathcal L'_k)_k$.
\end{exem}

\begin{exem}
Let $M:= \mathbb S^1\times \mathbb S^1\times \mathbb R^2$ be the product of the 2-torus with the real plane. Let $\alpha\in \mathbb S^1$.
\[\mathrm{Let} \; f:\;  M:= \mathbb S^1\times \mathbb S^1\times \mathbb R^2\rightarrow M\]
\[(\theta,\phi, x,y)\mapsto (2\theta,\phi+\alpha, 0, 10 y).\]
We notice that $f$ is $4$-normally hyperbolic at the torus $\mathbb T^2:=\mathbb S^1\times \mathbb S^1\times \{0\}$. However it is not injective restricted to this torus, and so we cannot apply Theorem \ref{preth2}. 

That is why we consider the Smale's solenoid $\tilde{\mathbb S}$ which projects onto $\mathbb S^1$ by sending the points of a same stable manifold to a same point of $\mathbb S^1$. Let $\pi$ be this projection.
We endow $\tilde {\mathbb  T}^2:=\tilde{\mathbb S}\times \mathbb S^1$ with the 2-dimensional lamination structure whose leaves are the product of the unstable manifolds with the circle $\mathbb S^1$. 
We notice that the map $i:\; (x,\phi) \in \tilde{\mathbb S}\times \mathbb S^1\mapsto (\pi(x),\phi, 0,0)\in M$ is an immersion of this lamination onto the torus  $\mathbb T^2$. 
 Also the dynamics $f_{|\mathbb T^2}$ lifts to the product $\tilde f$  of the usual dynamics of the solenoid with the rotation of angle $\alpha$. In other words the following diagram commutes:
\[\begin{array}{rcccl}
     &              &          f   &                             &\\
     & \mathbb T^2\subset  M& \rightarrow                &  \mathbb T^2\subset  M&\\
    i& \uparrow        &                                &\uparrow   &i\\
     &\tilde {\mathbb T}^2& \rightarrow                    &\tilde {\mathbb T}^2& \\
     &                                  &  \tilde f                       &                                   &\\  
\end{array}\]
Theorem \ref{preth2} implies that the immersed lamination $\tilde{\mathbb T}^2$ is $C^4$-persistent. 
In particular, for any $C^4$-perturbation $f'$ of $f$, there exists an immersion $i'$ $C^4$-close to $i$ such that $f'$ sends each leaf of $\tilde{\mathbb T}^2$ immersed by $i'$ to the immersion by $i'$ of the image by $\tilde f$ of this leaf.

Actually one can show that the torus $\mathbb T^2$ is not persistent.

\end{exem}

I thank J.-C. Yoccoz who suggested that I study the persistence of complex laminations. I am also grateful to C. Bonatti wonder me about the persistence of normally hyperbolic laminations by endomorphisms, by looking at the preorbits spaces. I thank also D. Varolin, J. Milnor, M. Lyubich, J. Kahn and S. Bonnot for many discussions and valuable suggestions.

\section{Geometry of lamination}
\subsection{Definitions}Throughout this chapter, $r$ refers to a fixed positive integer or to ${\mathcal H}$. The vector space $\mathbb K$ refers to the real line $\mathbb R$ when $r$ is an integer or to $\mathbb C$ when $r$ is ${\mathcal H}$.  For our purpose it is convenient to denote by $C^\mathcal H$ the class of holomorphic maps.
 
Let us consider  a locally compact and second-countable metric space $L$ covered by open sets $(U_i)_i$, called
 \emph{distinguished open sets}, endowed with homeomorphisms $h_i$ from $U_i$ onto $V_i\times T_i$, where $V_i$ is an open set of $\mathbb K^d$ and $T_i$ is a metric space.
 
 We say that the \emph{charts} $(U_i, h_i)_i$ define a $C^r$-\emph{atlas} of a lamination structure on $L$ of dimension $d$ if the \emph{coordinate change} $h_{ij}=h_j\circ h_i^{-1}$ can be written in the form
 \[h_{ij}(x,t)=(\phi_{ij}(x,t),\psi_{ij}(x,t)),\]
 where $\phi_{ij}$ takes its values in $\mathbb K^d$, $\psi_{ij}(\cdot,t)$ is locally constant for any $t$ and: \begin{itemize}
 \item[-] if $r$ is an integer then the partial derivatives $(\partial_x^s \phi_{ij})_{s=1}^r$ exist and are continuous on the domain of $\phi_{ij}$,
 \item[-]  if $r$ is ${\mathcal H}$ then $(\phi_{ij}(\cdot ,t))_t$ is a continuous family of complex analytic maps. 
\end{itemize} 
 
 A $(C^r)$-\emph{lamination} is a metric space $L$ endowed with a maximal $C^r$-atlas $\mathcal{L} $.

 A \emph{ plaque } is a subset of $L$ which can be written in the form $h_i^{-1}(V_i^0\times\{t\})$, for a  chart $h_i$, a point $t\in T_i$, and a connected component $V_i^0$ of $V_i$. A plaque that contains a point $x\in L$ will be denoted by  $\mathcal{L}_x$; the union of the plaques containing $x$ and of diameter less than $\epsilon>0$ will be denoted by $\mathcal L_x^\epsilon$. As the diameter is given by the metric of $L$, the set $\mathcal L_x^\epsilon$ is -- in general -- not homeomorphic to a manifold. The \emph{leaves } of $\mathcal L$ are the smallest subsets of $L$ which contain any plaque that intersects them.

 
 If $V$ is an open subset of $L$, the set of the charts $(U,\phi)\in\mathcal{L}$ such that $U$ is included in $V$, forms a lamination structure on $V$, which is denoted by $\mathcal{L}_{|V}$. A subset $P$ of $L$ is \emph{saturated} if it is a union of leaves. An \emph{admissible}\label{def adm} subset  $A$ of $L$ is a locally compact and saturated subset of the restriction of the lamination to some open subset. 


\begin{Examples}
\begin{itemize}
    \item[-] A manifold of dimension $d$ is a lamination of the same dimension.
    \item[-] A $C^r$-foliation on a connected manifold induces a $C^r$-lamination structure.
    \item[-] A locally compact and second-countable metric space defines a lamination of dimension zero.
    \item[-] The Smale solenoid attractor is endowed with a structure of laminations whose leaves are the stable manifolds.
    \item[-] Let $M$ be the cylinder $\mathbb S^1\times \mathbb R$. 
    Let $\pi$ be the canonical projection of $\mathbb R$ onto $\mathbb S^1\cong \mathbb R/\mathbb Z$. 
    \[\mathrm{Let}\; L:= \Big\{({\theta},y)\in M: \; y= \arctan(\overline{\theta}), \; \pi(\overline{\theta})=\theta\Big\}\cup \mathbb S^1\times \{-\pi/2,\pi/2\}.\]  
  The compact space $L$ is canonically endowed with a $1$-dimensional lamination structure which consists of the leaves $\mathbb S^1\times \{-\pi/2\}$,  $\mathbb S^1\times \{\pi/2\}$, and a last one which spirals down to these two circles.
    \item[-] The stable foliation of an Anosov $C^r$-diffeomorphism defines a $C^r$-lamination structure whose leaves are the stable manifolds. \end{itemize}
\end{Examples}
\begin{Property} If $(L, \mathcal{L})$ and $(L',\mathcal{L}')$ are two laminations, then $L\times L'$ is endowed with the lamination structure whose leaves are the product of the leaves of $(L, \mathcal{L})$ with the leaves of  $(L',\mathcal{L}')$. We denote this structure by $\mathcal{L}\times\mathcal{L}'$.\end{Property}

\subsection{Morphisms of laminations }
A \textit{$C^r$-morphism (of laminations)} from $(L,\mathcal{L})$ to $(L',\mathcal{L}')$ is a continuous map $f$ from $L$ to $L'$ such that, seen via charts $h$ and $h'$, it can be written in the form:
\[h'\circ f\circ h^{-1} (x,t)= (\phi (x,t),\psi(x,t)),\]
 where $\phi$ takes its values in $\mathbb K^{d'}$, $\psi(\cdot , t)$ is locally constant and: \begin{itemize}
 \item[-] if $r$ an integer, then  $\partial_x^s \phi$ exists and is continuous on the domain of $\phi$, for all $s\in\{1,\dots ,r\}$,
 \item[-] if $r$ is ${\mathcal H}$, then $(\phi(\cdot,t))_t$ is a continuous family of complex analytic maps.\end{itemize} 
 
If moreover the linear map $\partial_x \phi(x,t)$ is always one-to-one (resp. onto), we will say that $f$ is an \emph{immersion (of laminations)} (resp. \emph{submersion}).

An \emph{ isomorphism (of laminations)} is a bijective morphism of laminations whose inverse is also a morphism of laminations.

 An \emph{embedding (of laminations)} is an immersion which is a homeomorphism onto its image.
 
 An \emph{endomorphism of $(L,\mathcal{L})$} is a morphism from $(L,\mathcal{L})$ into itself.

 We denote by:
 \begin{itemize}
\item[-] $Mor^r(\mathcal{L},\mathcal{L}')$ the set of the $C^r$-morphisms from $\mathcal{L}$ into $\mathcal{L}'$,
\item[-] $Im^r(\mathcal{L},\mathcal{L}')$ the set of the $C^r$-immersions from $\mathcal{L}$ into $\mathcal{L}'$,
\item[-] $Emb^r(\mathcal{L},\mathcal{L}')$ the set of the $C^r$-embeddings from $\mathcal{L}$ into $\mathcal{L}'$,
\item[-] $End^r(\mathcal{L})$ the set of the $C^r$-endomorphisms of $\mathcal L$.\end{itemize}

  We denote by $T \mathcal{L}$ the vector bundle over $L$, whose fiber at $x\in L$, denoted by $T_x \mathcal{L}$, is the tangent space at $x$ to its leaf. The differential $Tf$ of a morphism $f$ from $\mathcal L$ into $\mathcal L'$ is the bundle morphism from $T\mathcal L$ into $T\mathcal L'$ over $f$, induced by the differential of $f$ along the leaves of $\mathcal L$.  
\begin{rema} If $M$ is a manifold, we notice that $End^r(M)$ denotes the set of $C^r$-maps from $M$ into itself, possibly non-bijective and possibly with singularities.\end{rema}

\begin{exem} \label{lamhyper}
Let $f$ be a $C^r$-diffeomorphism of a manifold $M$. Let $K$ be a hyperbolic compact subset of $M$. Then, the union $W^s(K)$ of the stable manifolds of points in $K$ is a $C^r$-lamination $(L,\mathcal L)$ immersed injectively.

Moreover, if every stable manifold does not accumulate on $K$, then $(L,\mathcal L)$ is a $C^r$-embedded lamination.
 \end{exem}

\begin{proof} See \cite{PB1} Example 1.1.5.\end{proof}

\subsection{Riemannian metric on a lamination}
A \emph{Riemannian metric} $g$ on a $C^r$-lamination $(L,\mathcal L)$ is an inner product $g_x$ on each fiber $T_x\mathcal L$ of $T\mathcal L$, which depends continuously on the base point $x$.
It follows from the local compactness and the second-countability of $L$ that any lamination $(L,\mathcal{L})$ can be endowed with some Riemannian metric.

A Riemannian metric induces -- in a standard way -- a metric on each leaf.
For two points $x$ and $y$ in a same leaf, the distance between $x$ and $y$ is defined by:
\[d_g(x,y)=\inf_{\{\gamma\in Mor([0,1],\mathcal L);\gamma(0)=x,\gamma(1)=y\}}\int_0^1\sqrt{g(\partial_t\gamma(t),\partial_t\gamma(t))},dt.\]

\subsection{Equivalent Classes of morphisms}\label{Top}
We will say that two morphisms $f$ and $f'$ in $Mor^r(\mathcal{L},\mathcal{L}')$ (resp. $Im^r(\mathcal{L}, \mathcal{L}')$ and $End^r(\mathcal{L}))$ are \emph{equivalent} if they send each leaf of $\mathcal L$ into the same leaf of $\mathcal L'$. The equivalence class of $f$ will be denoted by $Mor^r_f (\mathcal{L},\mathcal{L}')$
(resp. $Im_f^r (\mathcal{L},\mathcal{L}')$ and $End_f^r(\mathcal{L})$).

Given a Riemannian metric $g$ on $(L',\mathcal L')$, we endow the equivalence class with the $C^r$-compact-open topology. Let us describe elementary open sets which generate this topology.

Let $K$ be a compact subset of $L$ such that $K$ and $f(K)$ are included in distinguished open subsets endowed with charts $(h,U)$ and $(h',U')$. We define $(\phi,\psi)$ by $ h'\circ f\circ h^{-1}= (\phi ,\psi)$ on $h(K)$.

Let $\epsilon>0$. The following subset is an elementary open set of the topology:
\[\Omega:=\Big\{f'\in Mor^r_f (\mathcal{L} ,\mathcal{L}')\;:\;
 f'(K)\subset U' \; \mathrm{and \; there\; is } \; \phi'\;\mathrm{s.t.}\; \quad\]\[h'\circ f'\circ h^{-1}=(\phi' ,\psi)
 \; \mathrm{and}\; \max_{h(K)}\big(\sum_{s=1}^r\|\partial_x^s\phi-\partial_x^s\phi'\|\big)<\epsilon\Big\}.\]
with the convention that the sum $\sum_{s=1}^{\mathcal H}$ is zero. Thus for $r=\mathcal H$, this topology is the $C^0$-compact-open topology induced by the metrics of the leaves.

For any  manifold $M$, each of the spaces $Im^r(\mathcal L,M)$, $Emb^r(\mathcal L,M)$ and $End^r(M)$ contains a unique equivalence class. We endow these spaces with the topology of their unique equivalence class.

In particular the topology on $C^r(M,M)=End^r(M)$ is the (classical) $C^r$-compact-open topology.

\subsection{Tubular neighborhood of an immersed lamination}
Let $r\in [\![ 1,\infty[\![ \cup \{\mathcal H\}$. 
Let $(L,\mathcal L)$ be a lamination $C^r$-immersed by $i$ into a $C^r$-manifold $M$. 

 Via $i$, we can identify the bundle $T\mathcal L$ over $L$ to a subbundle of $i^*TM\rightarrow L$ whose fiber at $x\in L$ is $T_{i(x)}M$.  
The quotient bundle 
\[\pi\; :\; F:= i^*TM/T\mathcal L\rightarrow L\]
is called {\it the normal bundle} of the immersed lamination $(L,\mathcal L)$. Let $d$ be the dimension of $\mathcal L$ and let $n$ be the dimension of $M$. Thus the dimension of the fibers is $n-d$. 

A \emph{$C^r$-tubular neighborhood} of the immersed lamination $(L,\mathcal L)$ is the data of a $C^r$-lamination structure $\mathcal F$ on a neighborhood $F'$ of the $0$-section, with a $C^r$-immersion $I$ from $(F',\mathcal F)$ into $M$, such that:
\begin{itemize} 
\item[-] The preimage by $\pi_{|F'}$ of the plaques of $\mathcal L$ are plaques of $\mathcal F$. Hence the dimension of $\mathcal F$ is $n$.

\item[-] the restriction $\pi_{|F'}$ is a $C^r$-submersion,
  \item[-] $I\circ 0_F=i$.
\end{itemize}

We denote such a tubular neighborhood by $(F,\mathcal F,I,\pi)$.

\begin{prop}\label{exist:tubu} Every $C^r$-immersed lamination has a tubular neighborhood, when $r$ is a positive integer.\end{prop}

\begin{rema} The above proposition is not always true in the holomorphic case. For example, the equation $y^2=4X^3+aX+b$ defines a torus of $\mathbb C\mathbb  P^2$, whose modulus varies with $(a,b)\in \mathbb C^2$. Such tori depend (differentially) smoothly on $(a,b)$ but are not all biholomorphic. Thus, such tori cannot have a holomorphic tubular neighborhood.\\
Nevertheless, all Stein submanifolds (see \cite{Siu}) among others are endowed with an analytic tubular neighborhood.\end{rema}  
\begin{proof}[Proof of Proposition \ref{exist:tubu}] We showed in \cite{PB1} appendix A.2.1 that there exists a $C^r$-lifting $N$ of $i$ in the Grassmannian of $(n-d)$-planes of $TM$ such that, for every $x\in L$:
\[N(x)\oplus Ti(T_x\mathcal L)=T_{i(x)}M.\]
As the above sum is direct, we shall canonically identify the fiber $F_x$ of $F$ at $x$, with the subspace $N(x)$. Let $g$ be the smooth Riemannian metric on $M$, and $\exp$ be the exponential map associated to $g$.\[\mathrm{Let }\; I\; :\; F\rightarrow M\]
\[(x,u)\mapsto \exp_x(u'),\]
where $u'$ is a representative of $u$ in $N(x)$.\\
\indent Such a map $I$ is well defined on a neighborhood of the zero section of $F$. The existence of a suitable laminar structure $\mathcal F$ follows from Lemma \ref{L-fibre}. 
\end{proof} 

\section{Persistence of laminations}
\subsection{Preserved laminations}
 A lamination $(L,\mathcal L)$ embedded by $i$ into a manifold $M$ is \emph{preserved} by an endomorphism $f$ of $M$ if
 each embedded leaf of $\mathcal L$ is sent by $f$ into an embedded leaf of $\mathcal L$.
 
This is equivalent to suppose the existence of an endomorphism $f^*$ of $(L,\mathcal L)$ such that the following diagram commutes:
\begin{equation}\label{commute}\begin{array}{rcccl}
&&f&&\\
&M&\rightarrow&M&\\
i&\uparrow&&\uparrow&i\\
&L&\rightarrow&L&\\
&&f^*&&\end{array}\end{equation}

The endomorphism $f^*$ is called the \emph{pullback of $f$ via $i$}.

When the lamination is only immersed by $i$, these two definitions are not equivalent.

A lamination $(L,\mathcal L)$ immersed by $i$ into a manifold $M$ is \emph{preserved} by an endomorphism $f$ of $M$ if there exists a \emph{pull back of $f$ in $(L,\mathcal L)$ via $i$}. That is an endomorphism $f^*$ of $(L,\mathcal L)$ such that the diagram (\ref{commute}) commutes. We notice that the pull back $f^*$ has at least the minimum of the regularities of $f$ and $i$, as soon as $f^*$ exists and is continuous.

The leaves of a lamination $(L,\mathcal L)$ immersed by $i$ into a manifold $M$ are \emph{preserved} by an endomorphism $f$ of $M$ if the immersion of each leaf of $\mathcal L$ is sent by $f$ into the immersion of a leaf of $\mathcal L$.

Clearly, if $f$ preserves an immersed lamination, then it preserves its leaves. We give in \cite{PB1}  two examples of diffeomorphisms preserving the leaves of an immersed lamination but not the lamination.

\subsection{Persistence of laminations}
Let us fix $r\in[\![1,\infty[\![\cup \{\mathcal H\}$. 

Let $(L,\mathcal L)$ be a lamination $C^r$-embedded by $i$ into a manifold $M$.  Let $f$ be a $C^r$-endomorphism of $M$ which preserves $\mathcal L$. The \emph{embedded lamination} $(L,\mathcal L)$ is \emph{$C^r$-persistent} if for any endomorphism $f'$ $C^r$-close to $f$, there exists an embedding $i'$ $C^r$-close to $i$ such that $f'$ preserves the lamination $(L,\mathcal L)$ embedded by $i'$ and such that 
each point of $i'(L)$ is sent by $f'$ into the image by $i'$ of a small plaque containing $f(x)$. 
This implies that the pullback $f'^*$ of $f'$ is equivalent and $C^r$-close to the pullback $f^*$ of $f$.

Let $(L,\mathcal L)$ be a lamination immersed by $i$ into a manifold $M$.  Let $f$ be a $C^r$-endomorphism of $M$ which preserves $\mathcal L$. Let $f^*$ be a pull back of $f$ in $(L,\mathcal L)$.
The \emph{immersed lamination} $(L,\mathcal L)$ is $C^r$-\emph{persistent} if for any endomorphism $f'$ $C^r$-close to $f$, there exists an immersion $i'$ $C^r$-close to $i$, such that $f'$ preserves the lamination $(L,\mathcal L)$ immersed by $i'$, and has a pullback $f'^*$ in $(L,\mathcal L)$ equivalent and $C^r$-close to $f^*$. In other words, for  every $f'\in End^r(M)$ close to $f$ there exists $i'\in Im^r(\mathcal L,M)$ and $f'^*\in End_{f^*}^r(\mathcal L)$ close to respectively $i$ and $f^*$ such that the following diagram commutes:
\[\begin{array}{rcccl}
&&f'&&\\
&M&\rightarrow&M&\\
i'&\uparrow&&\uparrow&i'\\
&L&\rightarrow&L&\\
&&f'^*&&\end{array}\]

In the above definitions,  the topologies of the spaces $End^r(M)$, $Im^r(\mathcal L,M)$, $Emb^r(\mathcal L,M)$ and $End_{f^*}^r(\mathcal L)$ are described in section \ref{Top}.

\subsection{Normal hyperbolicity}
Let $(L,\mathcal L)$ be a lamination imbedded by $i $ into be a Riemannian manifold $(M,g)$. Let $f$ be an endomorphism of $M$ which preserves the immersion $i$ of $(L,\mathcal L)$. Let $f^*$ be a pullback of $f$. 
 
We identify, via the immersion $i$, the bundle $T\mathcal L\rightarrow L$ to a subbundle of $\pi\;:\; i^*TM\rightarrow L$. We remind that the fiber of $i^*TM$ at $x\in L$ is $T_{i(x)}M$. 

\begin{defi}\label{def2.1} For every $r\ge 0$, we say that $f$ \emph{is $r$-normally hyperbolic to the lamination $(L,\mathcal L)$ (immersed by $i$ over $f^*$)}, if the following conditions hold:
\begin{itemize}
\item there exists a splitting $i^*TM= E^s\oplus T\mathcal L\oplus E^u$ such that $E^s$ is $i^*Tf$-stable: $Tf$ sends $E^s_x$ into $E^s_{f^*(x)}$, for every $x\in L$,
\item there exist a norm on $i^*TM$ and $\lambda<1$ such that, for any unit vectors $v_s\in E^s_x$, $v_c\in E^c_x$ and $v_u\in E^u_x$, we have: 
\[\|Tf(v_s)\|<\lambda\cdot \min(1,\|Tf(v_c)\|^r),\]
\[\|p_u \circ Tf(v_u)\|>\lambda^{-1}\cdot \max (1,\|Tf(v_c)\|^r),\]
where $p_u$ is the projection of $i^*TM_{f^*(x)}=T_{i\circ f^*(x)}M$ onto $E^u$, parallelly to $E^s\oplus T\mathcal L$.
\end{itemize}
When  $E^u$ has dimension zero, we say that $f$ $r$-\emph{normally contracts} $(L,\mathcal L)$.  When $E^s$ has dimension zero, we say that $f$ $r$-\emph{normally expands} $(L,\mathcal L)$.   
\end{defi}

\begin{rema} Usually, one defines the normal hyperbolicity without changing the metric, but by looking some iterate of the dynamics. The definition stated above is actually more general. It is even strictly more general when the lamination is non-compact. Both of them are strictly more general that the one of HPS, even if their proof works as well for the usual definition.\end{rema}

\begin{rema} When $f^*$ is bijective, the subbundle $E^u$ of $i^*TM$ can be chosen $i^*Tf$-invariant: $Tf(E^u_x)=E^u_{f^*(x)}$, for every $x\in L$.\end{rema} 

\begin{rema} For the non-compact case, sometime we need to generalize the concept of normal hyperbolicity as follow. Let $f$ be an endomorphism of a manifold $M$ and let $(L,\mathcal L)$ be a lamination immersed by $i$ into $M$. We say that $f$ is \emph{$r$-normally hyperbolic} to $(L,\mathcal L)$
over a morphism $f^*$ from the restriction of $(L,\mathcal L)$  to an open subset $D\subset L$ to $\mathcal L$ if:
\begin{itemize} \item the following diagram commutes
$\begin{array}{rcccl}
&&f&&\\
&M&\rightarrow&M&\\
i&\uparrow&&\uparrow&i\\
&L'&\rightarrow&L&\\
&&f^*&&\end{array}$,
\item the properties of Definition \ref{def2.1} are satisfied.
\end{itemize}
\end{rema}

\begin{Property}\label{section:grass}
Let $(L,\mathcal L)$ be a lamination immersed by $i$ into a manifold $M$ and $1$-normally hyperbolic.
Then, for all $x$, $y\in L$ with the same images by $i$, the spaces $E^s_x$ and $E^s_x\oplus T_x \mathcal L$ have the same images by $Ti$ as respectively $E^s_y$ and $E^s_y\oplus T_y\mathcal L$.

Thus we can denote by respectively $E^s_{i(x)}$ and $(T\mathcal L\oplus E^s)_{i(x)}$ the subspaces of $T_{i(x)}M$ corresponding to $E^s_x$ and  $T_x\mathcal L\oplus E^s_{x}$, via the canonical identification of $i^*TM_x$ with $T_{i(x)}M$.
 
Moreover, for every compact subset $K$ of $L$, the section of the Grassmannian $z\in i(K)\rightarrow (E^s\oplus T\mathcal L)_z$ is continuous.
\end{Property}
\begin{proof} See \cite{PB1} Property 2.1.7.\end{proof} 

\begin{rema}
There exist compact laminations $(L,\mathcal L)$ that are immersed by some $i$, normally contracted (and so normally hyperbolic), with bijective pullback, but for which there are two different points $x$ and $y$ in $L$ whose images by $i$ are the same whereas  the images of $T_x\mathcal L$ and $T_y\mathcal L$ by $Ti$ are different.

For instance, the following map of the cylinder $\mathbb R\times \mathbb S^1$:
\[\phi\; : \; (x,\theta) \mapsto \Big(\frac{x}{2}+\frac{1}{2}\cdot \sin(\theta), 2\cdot \theta\Big)\]
preserves such an immersed solenoid around $\{0\}\times \mathbb S^1$; we represent it in figure \ref{attracteur} by identifying the cylinder to the punctured plane.       
\end{rema}

\begin{figure}[h]
    \centering
        \includegraphics{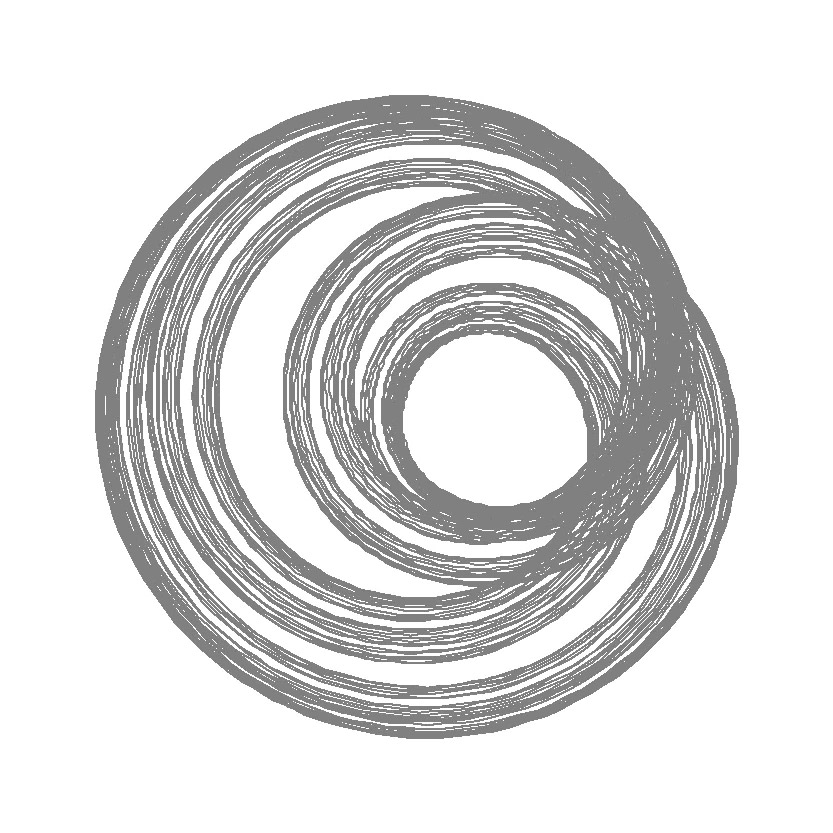}
    \caption{An attractor of the punctured plane.}
    \label{attracteur}
\end{figure}

%

\begin{exem}[Suspension of a holomorphic diffeomorphism]\label{suspension} Let $H\;: \; (z,z')\in \mathbb C^2\rightarrow (z^2+c+bz',z)\in \mathbb C^2$ be a complex H\'enon map which lets invariant a compact subset $K$ of  $\mathbb C^2$ ($H(K)=K$) and is hyperbolic at $K$. Such a compact subset can be the Julia set of $K$ when $|c| > 2(|b| + 1)^2$ ($H$ is a horseshoe). Let $M$ be the quotient of $\mathbb C\setminus \{0\}\times \mathbb C^2$ by the proper, discontinuous and fixed-point-free group action:
\[\mathcal G:= \Big\{g_n\;:\; g_n(x,z,z')\mapsto (2^n\cdot x,H^n(z,z')),n\in \mathbb Z\Big\}.\]
The complex $1$-lamination $(\tilde L,\tilde {\mathcal L}):=\mathbb C\setminus \{0\} \times K$ is $\mathcal G$-equivariant and so its image by the projection $\mathbb C\setminus \{0\}\times \mathbb C^2\rightarrow M$ is a compact (complex) lamination $(L,\mathcal L)$. The endomorphism $\tilde h:= (x,z,z')\mapsto (a\cdot x, H(z,z'))$ preserves and is normally hyperbolic to $(\tilde L,\tilde {\mathcal L})$, for any $a\in\mathbb C\setminus \{0\}$. Moreover, $\tilde h$ is $\mathcal G$-equivariant. Thus $\tilde h$ defines a diffeomorphism $h$ of $M$ which is $r$-normally hyperbolic to $(L,\mathcal L)$, for any $r\ge 0$. By Theorem \ref{preth2}, this (immersed) lamination is $C^r$-persistent for any $r\ge 1$.
\end{exem}

All of the above laminations are persistent as embedded laminations since the dynamics are plaque-expansive. Let us recall the definition of plaque-expansiveness.            
 
\subsection{Plaque-expansiveness}\label{plaque-expansiveness}
\begin{defi}[Pseudo-orbit which respect $\mathcal L$] Let $(L,\mathcal L)$ be a lamination and $f$ be an endomorphism of $(L,\mathcal L)$. Let $\epsilon$ be a positive continuous function on $L$. An \emph{$\epsilon$-pseudo-orbit (resp. backward-pseudo-orbit, resp. forward-pseudo-orbit) which respects $\mathcal L$} is a sequence $(x_n)_{n}\in L^{\mathbb Z}$ (resp. $(x_n)_{n}\in L^{\mathbb Z^-}$, resp. $(x_n)_{n}\in L^{\mathbb N}$) such that, for any $n$, the point
$f(x_n)$ belongs to a plaque of $\mathcal L$ containing $x_{n+1}$ whose diameter is less than $\epsilon(x_{n+1})$.\end{defi}

\begin{defi}[Plaque-expansiveness] Let $\epsilon>0$. The endomorphism $f$ is \emph{$\epsilon$-plaque-expansive (resp. backward-plaque-expansive, resp. forward-plaque-expansive) at $(L,\mathcal L)$} if for any positive function $\eta$ less than $\epsilon$, for all $\eta$-pseudo-orbits $(x_n)_{n}$ and $(y_n)_{n}$ which respect $\mathcal L$, such that for any $n$
the distance between $x_n$ and $y_n$ is less than $\eta(x_n)$, then $x_0$ and $y_0$ belong to a same small plaque of $\mathcal L$. \end{defi}

\begin{rema} We remark that every map $f'\in End_{f}(\mathcal L)$ close enough to $f^*$ is plaque-expansive (resp. backward-plaque-expansive, resp. forward-plaque-expansive).\end{rema}

\begin{rema} Forward or backward plaque-expansiveness implies plaque-expansiveness. 
\end{rema}

 We do not know if normal hyperbolicity implies plaque-expansiveness, even when $L$ is compact. But in many cases this is true.  Let us recall the results on plaque-expansiveness in the diffeomorphism case.
 Let $f$ be a diffeomorphism which preserves and is normally hyperbolic to a compact lamination $(L,\mathcal L)$.
  If one of the following situations occur, then $f$ is plaque-expansive at $\mathcal L$.
\begin{itemize}
\item \cite{HPS} the leaves of the lamination are the fibers of a bundle,
\item \cite{HPS} the lamination is  locally a saturated subset of $C^1$-foliation,
\item \cite{RHU} if for any $\epsilon>0$, there exists $\delta>0$ such that for any $x\in L$, $f^{n}(\mathcal L_x^\delta)$ is included in $\mathcal L_{f^n(x)}^\epsilon$, for any $n\in \mathbb Z$ (or $n\ge 0$ in the normally expanded case).   
 \end{itemize} 
In \cite{PB1}, we establish a generalization of the last result in the endomorphism context: we suppose the lamination normally expanded. Then if there exist $\epsilon>0$ and $\delta>0$ such that:
\begin{itemize}\item for every $x$ the subset $\mathcal L_x^\epsilon$ is precompact in its leaf,
 \item for any $x\in L$, the map $f^{n}$ sends $\mathcal L_x^\delta$ into $\mathcal L_{f^n(x)}^\epsilon$, for any $n\ge 0$.\end{itemize}
Then the endomorphism $f$ is forward plaque-expansive.

\section{Main theorem on  differentiable persistence}
\subsection{Statement}
 The following is our main theorem on $C^r$-persistence of non compact laminations; it implies Theorems \ref{preth1} and \ref{preth2}. 
 \begin{theo}\label{Main}
Under the hypotheses of Theorem \ref{preth1} (resp. \ref{preth2}) except that the lamination $(L,\mathcal L)$ is not necessarily compact, for every precompact open subset $L'$ of $L$, for every $f'$ $C^r$-close to $f$ there exist an immersion $i(f')$ of $(L,\mathcal L)$ into $M$ satisfying:

\begin{enumerate}
\item there exists a $C^r$-morphism $f'^*$ from $(L',\mathcal L_{|L'})$ into $(L,\mathcal L)$ such that the following diagram commutes:
 \[\begin{array}{rcccl} &        &     f'     &        &\\
                        &M       &\rightarrow&M       &\\
                       i(f')&\uparrow&           &\uparrow&i(f')\\
                        &L'       &\rightarrow&L       &\\
                        &        &f'^*        &        &\end{array}\]   
\item the morphism $f'^*$ is equivalent and $C^r$-close to the restriction $f^*_{|L'}$ of a pullback $f^*$ of $f$; the immersion $i'$ is $C^r$-close and is equal to $i$ on the complement of a compact subset of $L$ independent of $f'$.
%
\item There exists $\epsilon>0$ such that for any $f'$-orbit $(y_n)_{n\ge 0}\in M^\mathbb N$ (resp. $(y_n)_{n}\in M^\mathbb Z$) which is $\epsilon$-close to the image by $i$ of an $\epsilon$-pseudo-orbit $(x_n)_n\in L'^{\mathbb N}$ (resp. $(x_n)_n\in L'^{\mathbb Z}$) of $f^*$, respecting the plaques, then there exists $z_0$ uniformly close to $x_0$ in its leaf sent to $y_0$ by  $i(f')$. 
\end{enumerate}       
\end{theo}
\begin{rema} Conclusion (3) implies the uniqueness of $(i(f'), f'^*)$ in Theorems \ref{preth1} and \ref{preth2}, up to reparametrization.
\end{rema}

\begin{rema} This theorem implies the persistence of immersed compact laminations stated in Theorems \ref{preth1} and \ref{preth2} by taking $L'=L$. The persistence of  compact embedded laminations, given in these same theorems, is consequence of Corollary \ref{plaque-exp} (see below).\end{rema}   

In the normally contracting case, we have the following similar theorem:

\begin{theo}\label{Maincontract}
Let $r\ge 1$. Let $(L,\mathcal L)$ be a lamination $C^r$-immersed by $i$ into a manifold $M$. Let $D$ be an open subset of $L$.  Let $f$ be a $C^r$-endomorphism of $M$ which $r$-normally contracts the lamination $(L,\mathcal L)$ over an injective, open immersion $f^*$ from $(D,\mathcal L_{|D})$ into $(L,\mathcal L)$.

Then for any open precompact subset $L'$ of $D$ and for every $f'$ $C^r$-close to $f$, there exist an immersion $i(f')$ of $(L,\mathcal L)$ into $M$ and a $C^r$-morphism $f'^*$ of $(L',\mathcal L_{|L'})$ into $(L,\mathcal L)$ satisfying conclusions 1 and 2 of Theorem \ref{Main}. Conclusion $(3)$ is replaced by 
\begin{itemize}\item[$(3')$] there exists a real number $\epsilon>0$ such that for any $f'$-preorbit $(y_n)_{n\le 0}\in M^{\mathbb Z^- }$ which is $\epsilon$-close to the image by $i$ of an $\epsilon$-pseudo-preorbit $(x_n)_{n\le 0}\in M^{\mathbb Z^-}$ of $f^*$ respecting the plaques, there exists then $z_0$ close to $x_0$ in its leaf sent to $y_0$ by $i(f')$.  
\end{itemize} 
\end{theo}

\begin{rema}
The proofs of these Theorems do not use the fact that $f$ and its perturbation are defined on the complement of a neighborhood of $i(cl(L'))$. Thus, these theorems remain true if $f$ is a $C^r$-map from a neighborhood $U$ of $i(cl(L'))$ into $M$, for any $C^r$-perturbation $f'\in C^r(U,M)$ (except that the domain of $f'$ in the diagram has to be replaced by $U$). Also in the normally expanded case, we do not need $f^*$ to be defined outside of any neighborhood of $cl(L')$.
 \end{rema}

\subsection{Non-compact Examples}
\begin{exem} Let $r\ge 1$ and let $f_1$ be a $C^r$ diffeomorphism of a manifold $N_1$. Let $\Lambda$ be a hyperbolic compact subset. Then, by Example \ref{lamhyper}, $W^s(\Lambda)$ is an invectively  $C^r$-immersed lamination $(L_1,\mathcal L_1)$ whose leaves are stable manifolds. 
Let $E_u$ be the unstable direction of $\Lambda$ and  let
\[m:=\min_{u\in E_u\setminus \{0\}} \frac{\|Tf(u)\|}{\|u\|}.\]
We may suppose $m>1$.

Let $M_2$ be a compact Riemannian manifold and let $f_2$ be a $C^r$-endomorphism of $M_2$ whose differential has norm less than $\sqrt[r]{m}$ (hence $f_2$ has possibly many singularities and is not necessarily bijective). 

Thus, the product dynamics $f:= (f_1,f_2)$ on $M:=M_1\times M_2$ $r$-normally expands the $C^r$-immersed lamination $(L,\mathcal L):=(L_1\times M_2, \mathcal L_1\times M_2)$ over a pullback $f^*$. Let $L'$ be a precompact, open subset  of $L$, whose closure is sent into itself by $f^*$ (there exists arbitrarily big such subsets). By  Theorem \ref{Main}, 
the $C^r$-immersed lamination $(L',\mathcal L_{|L'})$ is $C^r$-persistent, since $f'^*(L')$ is included in $L'$ for $f'$ $C^r$-close to $f$. Actually, this lamination is forward plaque-expansive and we will see that this implies its persistence as an embedded lamination.
\end{exem}

\begin{exem}\label{endomorphism:variables}
Let $f\; : (z,z')\in \mathbb C^2\mapsto (z(z+e^{i\theta}),4z')\in \mathbb C^2$, with $\theta\in \mathbb R$.  The endomorphism $1$-normally expands the immersed submanifold $\mathbb C \hookrightarrow \mathbb C\times \{0\}$  (for some special Riemannian metric). Let $L'$ be a neighborhood of the filled Julia set of the polynomial function $z\mapsto z(z+1)$. By Theorem \ref{Main},
for every $f'$ $C^1$-close to $f$, there exists a $C^r$-immersion $i(f')$ of $\mathbb C\times \{0\}$ into $\mathbb C^2$ such that any point $x\in L'$ has its image by $i(f')$ sent by $f'$ to the image by $i(f')$ of some point $y\in \mathbb C\times \{0\}$. 

Let $K_{f'}$ be the set consisting of the points of $\mathbb C^2$ with bounded $f'$-orbits. One easily shows that $K_{f'}$ is included in any fixed neighborhood $W$ of $L'\times \{0\}$, for $f'$ close enough to $f$. Thus, by conclusion (3) of Theorem \ref{Main}, $K_{f'}$ is included in $i'(L')$.

 It is well known that for some perturbation $f'$ of $f$, the subset $K_{f'}$ is a Cantor set.      
\end{exem}
\begin{exem}\label{julia:fibre}  
In Example \ref{suspension} we recalled that for a large open subset $L$ of parameters $(b,c)\in\mathbb C^2$, the H\'enon map
$H_{c,b}(z,z')\mapsto (z^2+c+b\cdot z',z)$
has an uniformly  hyperbolic horseshoe. This means that $H_{c,b}$ preserves a hyperbolic compact subset $K_{c,b}$ and its restriction to $K_{c,b}$ is conjugated to a shift map on the set $\Sigma_2:= \{0,1\}^{\mathbb Z}$. 

Let $\mathbb A$ be an annulus (with finite modulus) holomorphically immersed into $L$. Let $M$ be equal to the product of $\mathbb A$ with $\mathbb C^2$. Let $f\; :\; (t,z,z')\in M\mapsto \big(t,H_{t}(z,z')\big)\in M$. By  the hyperbolic continuation theorem (or Theorem \ref{defo-hyp}), the subset $\cup_{t\in \mathbb A} K_t$ is canonically endowed with a structure of complex one-dimensional lamination $(L,\mathcal L)$, whose structure is given by the hyperbolic continuation theorem. The diffeomorphism $f$ preserves this lamination and is $r$-normally hyperbolic to it, for any $r\ge 1$. We endow $(L,\mathcal L)$ with the holomorphic tubular neighborhood $(F,\mathcal F,I,\pi)$ induced by the bundle $\mathbb A\times \mathbb C^2\rightarrow \mathbb A$. 
Let $\mathbb A'$ be an open precompact subset of $\mathbb A$ preserved by any complex automorphism of $\mathbb A$. Thus $\mathbb A'$ is also an annulus. Let $L'$ be the open precompact set equal to $L\cap (\mathbb A'\times \mathbb C^2)$. 
By Theorem \ref{Main}, for any diffeomorphism $f'$ $C^r$-close to $f$, there exists an immersion $i(f')$ $C^r$-close to the canonical inclusion, respecting the fibers of $\mathbb A\times \mathbb C^2\rightarrow \mathbb A$, and such that any point $x\in L'$ has its image by $i(f')$ sent by $f'$ into the image by $i(f')$ of a small plaque of $f(x)$. 

If $f'$ is holomorphic then it is of the form $f'(t,z,z')=(g(t),H'_{t}(z,z'))$, with $g$ a complex automorphism of $\mathbb A$.  Thus $g(\mathbb A')$ is equal to $\mathbb A'$ and $f'^*(\mathbb A')$ is also equal to $\mathbb A$. Consequently,  $f'$ preserves the lamination $(L',\mathcal L_{|L'})$ $C^r$-immersed by $i(f')$.  
\end{exem}
\subsection{Plaque-expansiveness implies injectivity}
\begin{coro}\label{plaque-exp} Under the hypotheses of Theorem \ref{Main} in the normally hyperbolic case (resp. normally expanded case, resp. under the hypotheses of Theorem \ref{Maincontract}), if $i$ is an embedding and $f^*$ is plaque-expansive (resp. forward plaque-expansive, resp. backward plaque-expansive), then the restriction of $i(f')$ to $\Lambda^*:= \cap_{n\in \mathbb Z} f'^{*^{n}}(cl(L'))$ (resp. $\cap_{n\ge 0} f'^{*^{n}}(cl(L'))$, resp. $\cap_{n\le 0} f'^{*^{n}}(cl(L'))$) is a homeomorphism, for every $f'$ $C^r$-close to $f$.\end{coro}  
\begin{proof}
As the three proofs are very similar, we shall only show the one of Theorem \ref{Main} in the normally hyperbolic case.
As $\Lambda^*$ is compact, we can take the function $\epsilon$ (of the plaque-expansiveness definition) equal to the constant $\epsilon_0:= \min_{\Lambda^*} \epsilon$.

For $V_f$ sufficiently small, for any $f'\in V_f$, the endomorphism $f'^*$ sends every point $x\in L'$ into the plaque $\mathcal L^{\epsilon_0}_{f^*(x)}$ containing $f^*(x)$ and with diameter less than $\epsilon_0$.

On the other hand, the following map is continuous:

\[\phi\; :\; V_f\rightarrow \mathbb R\]
\[f'\mapsto \min \big\{d\big(i'(x),i'(y)\big),\; (x,y)\in \Lambda^{*^2}\; d(x,y)>\epsilon_0\big\}, \quad \mathrm{with}\; i':= i(f').\]

As $\phi$ is positive at $f$, by restricting $V_f$, we may suppose that $\phi$ is positive on $V_f$.

Consequently, for every $f'\in V_f$, if $x,y\in \Lambda^*$ are sent by $i'$ to a same point, then they are $\epsilon_0$-distant. By commutativity of the diagram of Theorem \ref{Main}, for all $n\ge 0$, $f'^{*^n}(x)$ and $f'^{*^n}(y)$ are sent by $i'$ to a same point and hence, are $\epsilon_0$-distant. As we firstly notice, $(f'^{*^n}(x))_n$ and $(f'^{*^n}(y))_n$ are two $\epsilon_0$-pseudo-orbits. Thus, by plaque-expansiveness, $x$ and $y$ belong to a same small plaque. As $i'$ is an immersion which depends continuously on $f'$, we may suppose that $i'$ is injective on such small plaques. Thus, $x$ and $y$ are equal and $i'$ is injective.  By compactness of $\Lambda^*$ and continuity of $i'$, the map $i'_{|\Lambda^*}$ is a homeomorphism onto its image. Thus, $i'$ is an embedding.\end{proof}

\section{Space of preorbits}
The following proposition provides a way to construct bijective pullbacks in order to use Theorems \ref{preth2} and \ref{Main} on persistence of normally hyperbolic laminations and Theorem \ref{Maincontract} on persistence of normally contracted laminations.
\begin{prop}\label{preorbit-space}
Let $r\in [\! [1,\infty[\! [ \cup \{{\mathcal H}\}$ and let $(L, \mathcal L)$ be a lamination of class $C^r$. Let $D$ be an open subset of $L$ and let $f^*$ be a proper, open $C^r$-immersion from $\mathcal L_{|D}$ into $\mathcal L$.
Let $\tilde L$ be the preorbits space of $f$:
\[\tilde L:= \{(x_i)_{i\ge 0}\in L^{\mathbb N}\; : \; f(x_{i+1})=x_{i},\; \forall i\}.\]
Then there exists a canonical $C^r$-lamination structure $\tilde{ \mathcal L}$ on $\tilde L$ such that the canonical projection $\pi \; : \; (x_i)_i\in \tilde L\mapsto x_0 \in L$ is an injective, open $C^r$-immersion and the map $(x_i)_i\in  \pi^{-1}(D)\mapsto (f(x_i))_i$ is a $C^r$-immersion onto $\tilde L$.\end{prop}
\begin{proof} Let $(U,\phi)\in \mathcal L$ be a chart of the form $\phi \; :\; U\rightarrow W\times T$, with $U$ a precompact, open subset of $L$, $T$ a locally compact metric space and $W$ a simply connected, open subset of $\mathbb K^n$.

We want to find a metric space $\tilde T$ and a homeomorphism $\tilde \phi\; :\; \pi^{-1}(U)\rightarrow  W\times  \tilde T$ such that the following diagram commutes:
\[\begin{array}{rccl}
     &\pi^{-1}(U)      &    \tilde\phi_1           &\\
\phi &\downarrow       &\searrow       &\\  
     &     U            & \rightarrow    &W\\
     &                  &\phi_1          &\end{array}\]
where $\phi_1$ and $\tilde \phi_1$ are the first coordinates of respectively $\phi$ and $\tilde \phi$.
     
Let $(x_n)_n\in \tilde L^\mathbb N$ be such that $x_0$ belongs to $U$. For each $n\le 0$, let $U_n$ be the union of plaques of $x_n$ included in $f^{-n}(U)$. We notice that $U_n$ is the connected component of $x_n$ in the intersection of $f^{-n}(U)$ with the leaf of $x_n$. Let us show that the restriction of 
$ f^n$ to $U_n$ is a homeomorphism onto $U_0$.

First of all, $f^{-n}(U)$ is an open subset of $L$ and so $U_n$ is a connected open subset of the leaf of $x_n$. 

Let us show that $f^n_{|U_n}$ is injective. Let $y,y'\in U_n$ be sent by $\phi_1\circ f^n$ to a same point $z\in U_0$. By connectedness of $U_n$, there exists a path in $U_n$ from $y$ to $y'$. As $U_0\cong W$ is simply connected, the image by $f^n$ of this path is homotopic to the trivial path $\{z\}$. As $ f^n$ is an immersion, we can pull back this homotopy to $U_N$. This shows that $y$ and $y'$ are equal.  

As the differential of $f^n$ along the leaves at each point of $U_n$ is an isomorphism, it follows from the global inversion theorem that the restriction  of $f^n$ to $U_n$ is a diffeomorphism onto its image which is open.    

To show that its image is equal to $U_0$, by connectedness of $U_0$, it only remains to prove that $f^n(U_n)$ is a closed subset of $U_0$.

Let $z\in cl(f^n(U_n))\cap U_0$. Then there exists a sequence $(y_k)_k\in U_n^\mathbb N$ such that the sequence $(f^n(y_k))_k$ converges to $z$. Let us admit that $U_n$ is precompact in the leaf of $x_n\in \mathcal L_{|D}$. Then we may suppose that $(y_k)_k$ converges to some $y$ in the closure of $U_n$ in the leaf of $x_n\in \mathcal L_{|D}$. By continuity of $f$, the image by $f^n$ of $y$ is $z$. By openness of $f$ and $U_0$, there exists a small plaque $\mathcal L_y$ containing $y$ in the leaf of $x_n$, which is sent into $U_0$. Since $(y_n)_n$ converges to $y$ in the leaf of $x_n$, the union of $U_n$ with $\mathcal L_y$ is connected. As this union is sent into $U_0$, this union is contained in $U_n$. Hence $y$ belongs to $U_0$ and $z$ belongs to $U_0$. Therefore, $f^n(U_n)$ is closed in $U_0$, and so $ f_{|U_n}$ is a homeomorphism onto $U_0$.

Let us show that $U_n$ is precompact in the leaf of $x_n$. We endow $\mathcal L$ and $\mathcal L_{|D}$ with respectively two (possibly different) complete metrics. Then, for these Riemannian distances, the bounded and closed subsets of leaves of $\mathcal L$ and $\mathcal L_{|D}$ are compact.  Thus, to show that $U_n$ is precompact in the leaf of $x_n$ in $\mathcal L_{|D}$, it is sufficient to show that $U_n$ is bounded for the complete  distance of $\mathcal L_{|D}$. We note that:

\[\diam_{\mathcal L_{|D}} U_{k+1}\le \sup_{U_{k+1}} \|Tf^{-1}\|\cdot \diam_{\mathcal L} U_k,\]
where $\diam_{\mathcal L_{|D}}$ and $\diam_{\mathcal L} $ refer to the diameters with respect to the complete metrics of  respectively $\mathcal L_{|D}$ and $\mathcal L$.

We remind that the finiteness of $\diam_{\mathcal L_{|D}} U_k$ is equivalent to its precompactness in a leaf of $\mathcal L_{|D}$, and so its precompactness in a leaf of $\mathcal L$. Moreover,  for $k\ge 1$, by precompactness of $U_0$ and properness of $f$, the subset $U_k$ is included in $f^{-k}(U_0)$ which is precompact in $D$. Thus $\sup_{U_k} \|Tf^{-1}\|$ is finite.  Using these two last remarks, an induction proves that, for $k\ge 1$, $U_k$ is precompact in a leaf of $\mathcal L_{|D}$.
    
Let $T_n$ be the set of the connected components of each intersection of $f^{-n}(U)$ with some leaf of $\mathcal L$. For instance, $T_0$ is the set $T$. As the elements of $T_n$ are homeomorphic to $W$ and hence are locally compact, we can endow $T_n$ with the following distance:
\[ d_{T_n}\;:\; (t,t')\in T_n\mapsto \sup_{x\in t} d(x,t')+ \sup_{x'\in t'} d(x',t),\]
 where the metric $d$ refers to the one of the metric space $L$.
 
We notice that $f$ induces a continuous map from $T_{n+1}$ to $T_n$, since $f$ is uniformly continuous on $U_{n+1}$. Thus, we shall consider the projective limit:
\[\tilde T:= \lim_{\leftarrow} T_n.\]

The topology of $\tilde T$ is given, for instance, by the following metric:
\[d_{\tilde T} ((t_n)_{n\ge0}, (t'_n)_{n\ge 0})= \sum_{n\ge 0} \frac{\min (d_{T_n} (t_n, t'_n), 1)}{2^n}.\]

For such a metric $\tilde T$ is locally compact.  

Therefore, the following map satisfies the requested properties:
\[\tilde \phi \;: \; \pi^{-1}(U)\rightarrow W\times \tilde T\]
\[\underline{x}=(x_n)_{n\ge 0} \mapsto (\tilde \phi_1 (x_0), (t_n(x))_n),\]
where $t_n(x)$ is the point of $T_n$ containing $x_n$.

The above map is well defined and is a homeomorphism since the projective limit $\tilde L$ has its topology given by the following metric:
\[\tilde d:\; \big( (x_n)_n, (y_n)_n\big) \mapsto \sum_{n\ge 0} \frac{\min\big( d(x_n,y_n), 1\big)}{2^n}\]
For such metric space  $(\tilde L, \tilde d)$ is secondly countable and locally compact. Also  the family $(\pi^{-1}(U), \tilde \phi)_U$ constructed as above forms an atlas $\tilde{\mathcal  L}$ of laminations on $(\tilde L, \tilde d)$.

\end{proof}   
%

\begin{exem}\label{henon+preorbit} Let $P$ be a polynomial function of $\mathbb C$. Let $C_0$ be the closure of the forward orbits of the critical values. Let $L_0$ be the nonempty open subset $\mathbb C\setminus C_0$ and $D_0$ be the preimage of $L_0$. The open set $L_0$ is canonically endowed with the complex one-dimensional lamination structure consisting of a single leaf. We remark that the restriction of $P$ to $D_0$ satisfies the hypotheses of the above proposition. Therefore, the set:
\[L:=\{(z_i)_{i\le 0}\in L^{\mathbb N}\;:\: P(z_{i-1})=z_i\}\]
is endowed with a structure of lamination $\mathcal L$.

 \[\mathrm{Let}\quad \pi\; :\; (z_i)_i\in L\mapsto z_0\in L_0 \quad \mathrm{and}\quad f^*\; :\; (z_i)_{i\le 0}\in D \mapsto (P(z_i))_{i\le 0}\in L, \]
 with $D$ be the preimage by $\pi$ of $D_0$.

We notice that the maps $\pi$ and $f^*$  are immersions for the structure induced by $\mathcal L$.  

The map $H\;:\; (z,z')\mapsto (P(z)+z',0)$ preserves and $r$-normally contracts the immersion $z\in L\mapsto (\pi(z),0)$ of $\mathcal L$ over $f^*$, for any $r\ge 1$. The strong stable direction is the tangent space at $z'=0$ of the foliation whose leaves are $\{(z,z'): \; z'= - P(z)\}$.

  Let $L'_0$ be a precompact open subset of $D_0$ and let $L'$ be the preimage of $L'_0$ by $\pi$. By Theorem \ref{Maincontract}, for any $C^r$-perturbation $H'$ of $H$, there exists an immersion $i(H')$ of $(L,\mathcal L)$ into $\mathbb C^2$, $C^r$-close to $i$, such that any point $z\in L'$ has its image by $i(H')$ sent by $H'$ to the image by $i(H')$ of a point $y\in L$. Moreover $y$ is close to $f^*(x)$ in its leaf of $\mathcal L$. 

We notice that for $P(z)=z^2+c$, possible perturbations are the complex H\'enon maps $H'\;:\; (z,z')\mapsto (z^2+z',bz)$, for small $b\in \mathbb C$. 
Also conclusion $(3)$ of Theorem \ref{Maincontract} implies that there exists an open neighborhood $W$ of $L'_0\times \{0\}$ in $\mathbb C^2$ such that any point of $W$ is in the image of $i(H')$ if it has a $H'$-preorbit in $W$.

This phenomenon was already known by Hubbard and Oberste-Vorth (\cite{HOV}) when $P$ is a hyperbolic map.\end{exem}

\section{Persistence of complex laminations}


Before giving examples of our Theorems \ref{defo-exp} and \ref{defo-hyp} on persistence of complex laminations by deformation, let us explain and illustrate how to use Theorems \ref{Main} and \ref{Maincontract} on persistence of non-compact laminations in the complex analytic context. 

The problem is to transform a $C^1$ immersion given by these theorems to a holomorphic immersion. Obviously the 
$J$-invariance of the tangent space to its immersed leaves is necessarily, where $J$ is the automorphism  of the tangent bundle of the manifold provided by its complex structure ($J^2=-1$). Conversely, the $1$-normal hyperbolicity of a lamination implies the $J$-invariance of the tangent space to its immersed leaves as stated in the following propositions:

\begin{prop}\label{J1} Under the hypotheses of Theorems \ref{Main} or \ref{Maincontract} with $r=1$, if $i$ and $f$ are moreover holomorphic, then for any holomorphic perturbation $f'$ of $f$, for any $\mathcal L$-admissible\footnote{See definition in Subsection \ref{def adm}} subset\footnote{See Corollary  \ref{plaque-exp} for a more precise definition of $\Lambda^*$} $A$ of $\Lambda^*:= \cap_n f^{*n}(L')$, the tangent space to the immersed lamination $(A,\mathcal L_{|A})$ by $i(f')$ is $J$-invariant.
\end{prop}
\begin{proof} 
We only show the normal contracting case (Theorem \ref{Maincontract}), the proofs of the normal expanding or hyperbolic cases are very similar.
We remind that a cone of a vector space $E$ is an open set invariant by nonzero and real scalar multiplications.

As the immersion by $i$ of $(L,\mathcal L)$ is $1$-normally contracted by $f$, for $f'$-close enough to $f$, 
there exists a family of cones $(C_x)_{x\in \Lambda^*}$ such that $C_x$ is formed by vectors of $T_{i(f')(x)}M$ satisfying the following conditions for every $x\in \Lambda^*$:
\begin{enumerate}
\item $\chi_x$ contains $Ti(f')(T_x\mathcal L)$,
\item $Tf'$ sends the closure of the cone $\chi_{f'^{*-1}(x)}$ into $\chi_x\cup\{0\}$,
\item $Ti(f')(T_x\mathcal L)$ is a maximal vector subspace included in $\chi_x$,

As $i$ is holomorphic and close to $i(f')$, we may also suppose that:
\item $\chi(x)$ contains a $J$-invariant $\dim(\mathcal L)$-plane.
\end{enumerate}
By using the Hilbert metric, one can show that assertions $2$ and $3$ imply that for any $x\in \Lambda^*$, the push-forward of $\chi$ by $f$ converges to $Ti(f')(T_x \mathcal L)$:
\begin{equation}\label{une:intersection} \cap_{n\ge 0} Tf^n(\chi_{f^{*-n}(x)})=Ti(f')(T_x\mathcal L).\end{equation}
One the other hand, as $f'$ is holomorphic, its tangent map $Tf$ commutes with $J$ ($Tf\circ J=J\circ Tf$). Thus assertions $2$ and $4$ imply that $\cap_{n\ge 0} Tf^n(\chi_{f^{*-n}(x)})$ contains a $J$-invariant $\dim(\mathcal L)$-plane. 
By equation (\ref{une:intersection}), such a plane is necessarily $Ti(f')(T_x\mathcal L)$.\end{proof} 

All the relevance of the above property follows from this basic proposition: 

\begin{prop}\label{J2} Let $(L,\mathcal L)$ be a lamination $C^1$-immersed into a complex manifold $(M,J)$.  
If for every $x\in L$, $Ti(T_x\mathcal L)$ is $J$-invariant, then there exists a unique complex structure $\mathcal L_{\mathcal H}$ on $L$, which is compatible with the $C^1$-structure $\mathcal L$ and such that $i$ is a holomorphic immersion.
\end{prop}
\begin{proof} 
\underline{Uniqueness}

Let $\mathcal L_{H}^1$ and $\mathcal L_{H}^2$ be two complex structures compatible with $\mathcal L$ such that $i$ is holomorphic for both structures. Let $x\in L$. Let $\phi_1:\; U_1\rightarrow V_1\times T_1$ and $\phi_2:\; U_2\rightarrow V_2 \times T_2$ be two charts of neighborhoods of $x$ in the structures $\mathcal L_{H}^1$ and $\mathcal L_{H}^2$ respectively. Let us denote by $J_1$ and $J_2$ the complex structures of $V_1\subset \mathbb C^d$ and $V_2\subset \mathbb C^d$ respectively.  

We have:
\[ T_xi\circ T\phi_1\circ J_1 = J\circ T_x i \circ T\phi_1 \quad \mathrm{and}\quad T_xi \circ T\phi_2\circ J_2 = J\circ T_x i \circ T\phi_2.\]

Consequently:
\[ (T_x \phi_2)^{-1} \circ T_x\phi_1 \circ J_1 = J_2.\]

In other worlds, the coordinate change $\phi_2^{-1}\circ \phi_1$ is holomorphic and so by maximality of $\mathcal L_H^1$  and $\mathcal L_H^2$, these two structures are equal.  

\noindent \underline{Existence}

Let us construct the holomorphic structure of $L$. Let $x\in L$. Via a chart of $M$, we may identify a neighborhood of $i(x)\in M$ to an open subset $V$ of $\mathbb C^n$, such that $i(x)$ is $0$ and $Ti(T_x \mathcal L)$ is $(\mathbb C^d\times \{0\})\cap V$. Let $\phi:\; U\rightarrow W\times T$ be a chart of a neighborhood of $x\in L$. We may suppose $U$ sufficiently small such that $i(U)$ is included in $V$ and such that for every $t\in T$, the set $i\circ \phi^{-1}(W\times \{t\})$ is the graph of a function from an open subset $W_t$ of $\mathbb C^d$ into $\mathbb C^{n-d}$.
By restricting $T$, we may suppose the existence of an open neighborhood $W'$ of $0$ in the intersection $\cap_{t\in T} W_t$.

Let $U'$ be the open subset of $L$ formed by the point $x'\in U$ such that the projection of $i(x')$ by $p:\; \mathbb C^n \rightarrow \mathbb C^d\times \{0\}$ is in $W'$. We notice that $U'$ contains $x$ and that 
\[\phi':\; U'\rightarrow W'\times T\]
\[ x'\mapsto (p\circ i(x), \phi_2(x'))\]
is a chart of $\mathcal L$, with $\phi_2$ the second coordinate of $\phi$.

Therefore one can easily show that the family of maps $\phi'$ constructed in such a way is a atlas of complex lamination.  
 \end{proof}

Therefore the complex structure of a $C^1$-persistent complex lamination may change, but for our purpose the complex structure remains the same if the lamination is endowed with a holomorphic tubular neighborhood:

\begin{prop}\label{J3}
Let $(L,\mathcal L)$ be a complex lamination immersed into a complex manifold $M$. Suppose that $(L,\mathcal L)$ is endowed with a holomorphic tubular neighborhood $(F,\mathcal F,I,\pi)$. Let  $i'$ be a $C^1$-immersion of $(L,\mathcal L)$ equal to the composition of $I$ with a $C^1$-section of $(F,\mathcal F)\rightarrow (L,\mathcal L)$. If for every $x\in L$, $Ti'(T_x\mathcal L)$ is $J$-invariant, then $i'$ is holomorphic.\end{prop}  

\begin{proof} Since $i'$ is differentiable, it is sufficient to show that for any $x\in L$, for any $u\in   T_x\mathcal L$:

\[ T_x i'(J u)= J\circ T_x i'(u). \]
Since $i'$ is equal to the composition of a $C^1$-section with a holomorphic tubular neighborhood $I$, there exists a plaque $\mathcal L_x$ of $x\in L$ and an open neighborhood $U$ of $i'(x)$ such that $U$ is biholomorphic to $\mathcal L_x\times V$, where $V$ is an open supset of $\mathbb C^p$ and in such a trivialization:
\[ i'_{|\mathcal L_x}:\; y\in \mathcal L_x\mapsto (y,\sigma(y)).\]
Since $i'(U)$ has its tangent space $J$-invariant, for any $u\in T_x U$ there exists $v\in T_x U$ such that:
\[ J\circ (u, T_x \sigma(u))= (v, T_x \sigma(v)).\]

This implies that $T_x \sigma$ commutes with $J$ and so that $T_x i'$ commutes with $J$.
\end{proof} 
 
Each complex lamination of Examples \ref{suspension},  \ref{endomorphism:variables}, \ref{julia:fibre} or
\ref{henon+preorbit} is endowed with a holomorphic tubular neighborhood. Also each of them is preserved by a holomorphic dynamics and satisfies the hypotheses of Theorems \ref{Main} or \ref{Maincontract}. Thus we can apply to each them the three above propositions.

 Let us proceed in detail for Example \ref{julia:fibre} about fibered hyperbolic horseshoes.  
In this example, we saw that the restriction of the lamination $(L,\mathcal L)$ to some subset $L':= L\cap (\mathbb A'\times \mathbb C^2)$ is $C^1$-persistent for any holomorphic perturbation $f'$ of the dynamics. In other words, there exists a $C^1$-embedding $i(f')$ of $(L',\mathcal L_{|L'})$ that $f'$ lets invariant \big($f'\big(i(f')(L')\big)=i(f')(L')\big)$. Thus $\Lambda^*$ is equal to $L'$, and the embedding $i(f')$ 
is holomorphic for a possibly non canonical complex structure on the lamination  $(L',\mathcal L_{|L'})$, by Propositions \ref{J1} and \ref{J2}. But, since $i(f')$ is the composition of $I$ with a $C^1$-section of a holomorphic tubular neighborhood $(F,\mathcal F,I,\pi)$ of $(L,\mathcal L)$ (induced by the bundle $\mathbb A\times \mathbb C^2\rightarrow \mathbb A$), this complex structure is equal to the initial one.

\begin{ques} Find a non trivial holomorphic perturbation of the suspended H\'enon map of Example \ref{suspension}.\end{ques}
 
Let us now give an example of complex submanifold, normally expanded by a biholomorphic map, which is differentially persistent, but for which the complex structure is not persistent. 
   

\begin{exem}[Hopf surface]\label{hopf}
Let $\tilde M:= \mathbb C\times  \big(\mathbb C^2\setminus \{(0,0)\}\big)$ and $\alpha\in \mathbb C$ with modulus in $]0,1[$. Let $g$ be the following biholomorphic map of $\tilde M$:
\[g\;:\; (t,z_1,z_2)\mapsto (t,\alpha\cdot z_1+t\cdot z_2,\alpha\cdot z_2).\]
We remark that the action on $\tilde M$ of the group $\mathcal G:= \{g^n,\; n\in \mathbb Z\}$ is proper, discontinuous and fixed-point-free on $\tilde M$. Thus the quotient $M:= \tilde M/\mathcal G$ is a complex manifold.

For $s\in \mathbb C$, the following holomorphic endomorphism of $\tilde M$:
\[\tilde f_s\; :\; (t,z_1,z_2)\mapsto (2(t-s),z_1,z_2)\]
is $\mathcal G$-equivariant. Thus there exists $f_s\in End^{\mathcal H}(M)$ such that the following diagram commutes:
\[\begin{array}{rcccl}
&\tilde M&\stackrel{\tilde f_s}{\rightarrow} &\tilde M&\\
\pi&\downarrow&&\downarrow& \pi\\
&M&\stackrel{f_s}{\rightarrow} &M&\end{array},\]
with $\pi\; :\; \tilde M\rightarrow M$ the canonical projection. We notice that the (Hopf) surface $M_s:= \pi\big(\{s\}\times (\mathbb C^2\setminus \{(0,0)\})\big)$ is normally expanded by $f_s$ since $\tilde f_s$ normally expands $\{s\}\times (\mathbb C^2\setminus \{(0,0)\})$. 

Actually, for every  $r\ge 1$ and every small $s$, $M_s$ is the unique surface $C^r$-close to $M_0$ that $f_s$ preserves. Nevertheless, for $s\not=0$, the Hopf surface $M_s$ is not biholomorphic to $M_0$ (see \cite{Koda}, ex. 3, p.23).
Thus, $M_0$ is a compact complex submanifold of $M$, which is $1$-normally expanded by $f_0$ and hence $C^1$-persistent  but not holomorphically persistent.

Nevertheless, by Theorem \ref{defo-exp} the torus $M_0$ is persistent by deformation.     
\end{exem} 

Let us now give an example of a normally expanded submanifold by an endomorphism, which is persistent by deformation, but not differentially neither biholomorphically persistent.
 
\begin{exem} In the above Example, 
the Hopf surface $M_0$ preserved and $0$-normally expanded by the endomorphism of induced by:
\[\tilde f_s\; :\; (t,z_1,z_2)\mapsto (4(t-s),10 z_1,10 z_2)\]
the submanifold $M_0$ is persistent by deformation, but not differentially nor biholomorphically persitent.\end{exem}
 
 Also Theorems \ref{defo-exp} and \ref{defo-hyp} enable us to construct new laminations as in the following example.
\begin{exem}
Let $P$ be a polynomial function of $\mathbb C$ which expands some repulsive compact subset $K$ of $\mathbb C$. For instance $P$ can be Collet-Eckmann or a hyperbolic quadratic polynomial function. Let $f\; :\; (z,(z_i)_i)\in \mathbb C\times \mathbb C^n\mapsto (P(z),0)$, for $n\in \mathbb N$. The $n$-dimensional embedded lamination $(L,\mathcal L)=K\times \mathbb C$ is preserved and $1$-normally expanded. 

As $K$ does not contain more than one point, $(L,\mathcal L)$ is endowed with a tubular neighborhood. Thus, by Theorem \ref{Main} and Propositions \ref{J1}, \ref{J2} and \ref{J3},   the $n$-dimensional lamination $(L',\mathcal L')=K\times B(0,R)$ is holomorphically persistent, for $R>0$ and with $B(0,R)$ the ball of $\mathbb C^n$ centered at $0$ of radius $R$.

This means that for every holomorphic map $f'$ close to $f$, there exists a holomorphic immersion $i(f')$ of $(L',\mathcal L')$ into $\mathbb C^{n+1}$, close to $i_{|L'}$ and such that for every $k\in K$, $f'$ sends $i(f')(\{k\}\times \mathbb C^n)$ into $i(f')(\{P(k)\}\times B(0,R))$.

On the other hand, by Theorem \ref{defo-exp}, the lamination $(L',\mathcal L')$ is persistent by deformation. This means that for a complex analytic family of holomorphic endomorphisms $(f_t)_{t\in B}$ of $\mathbb C^{n+1}$ such that $f_{t_0}=f$ for some $t_0\in B$, there exists a neighborhood $B_0$ of $t_0$ such that $B_0$ is included in $V_f$ and $i\; :\; (t,x)\in B_0\times L'\mapsto i(f_t)(x)\in \mathbb C^{n+1}$ is a holomorphic immersion for the a complex laminar structure $\mathcal D$ on $B_0\times L'$. By uniqueness of the structure $\mathcal D
$ is $B_0\times \mathcal L'$.

With $n=0$ and with $B$ the open set of parameters $c\in \mathbb C$ such that $P_c(z)=z^2+c$ is hyperbolic, the above discussion implies the well known result that $\cup_{c\in B} \{c\}\times J_c$ is endowed with a complex one-dimensional structure of lamination, where the Julia set $J_c$
 of $P_c$ is the transverse space.\end{exem}

\section{Proof of differentiable persistences of normally contracted laminations}
In this section we prove Theorem \ref{Maincontract}.
\subsection{The setting}
We fix $r\ge 1$. Let us denote by $K$ the closure of $L'$. 

Let $(F,\mathcal F,I,\pi)$ be a $C^r$-tubular neighborhood of the immersed lamination $(L,\mathcal L)$ into $M$.
We identify $L$ to the zero section of $F$. We endow $i^*TM$ with a norm satisfying the normal contraction inequality. We endow $F$ with the norm associating to a vector $v\in F$ the one of the projection of $T_0I(v)$ into $E^s$ parallelly to $T\mathcal L$ in $i^*TM$.

Let $d$ be the dimension of the leaves of $\mathcal L$. 

We will always suppose $\eta>0$ small enough such that for every $x$ in a neighborhood of $K\cup f^*(K)$:\begin{itemize}
\item[-]  the ball $B_{F_x}(0,\eta)$ centered at $0$ and with radius $\eta$ is diffeomorphically sent by $I$ to  a submanifold  $ F_x^\eta$,
\item[-]  the union $\mathcal L_x^\eta$ (resp. $\mathcal F_x^\eta$) of plaques of $\mathcal L$ (resp. $\mathcal F$) containing $x$ and with diameter less than $\eta$ is a plaque. \end{itemize}

\subsection{The algorithm}  
The idea of the proof is to use the following lemma which parametrizes the image by the dynamics of perturbations of the immersed lamination $\mathcal L$.

\begin{lemm}\label{lem6'}
For any small $\eta>0$, there exist an open neighborhood $V$ of $f^*(K)$, a neighborhood $V_i^0$ of $i\in Mor^r(\mathcal L,M)$, a neighborhood $V_f$ of $f\in End^r(M)$ and a continuous map
\[S^0\;:\; V_f\times V_i^0\rightarrow Im^r(\mathcal L_{|V},M)\]
satisfying:
\begin{enumerate}
\item the morphism $S^0(f,i)$ is equal to $i_{|V}$,

\item \label{preconc4'} for all $x\in V$, $i'\in V_i^0$ and $f'\in V_f$,
 the image of $i'(\mathcal L_{f^{*^{-1}}(x)}^{\eta})$ by $f'$ intersects transversally $F_{x}^{\eta}$ at a unique point $S^0(f',i')(x)$.
\end{enumerate}

\end{lemm}

\begin{proof}
Let us remind that $f^*$ is an injective immersion from $\mathcal L_{|D}$ into $\mathcal L$. Thus, we may suppose $\eta>0$ small enough such that $f'$ sends diffeomorphically $i'(\mathcal L_{f^{*-1}(x)}^\eta)$ to a submanifold transverse at a unique point $v$ to the submanifold $F_{x}^{\eta}$, for all $x$ close to $f^*(K)$, $i'$ $C^r$-close to $i$ and $f'$ $C^r$-close to $f$.

\begin{figure}[h]
    \centering
        \includegraphics{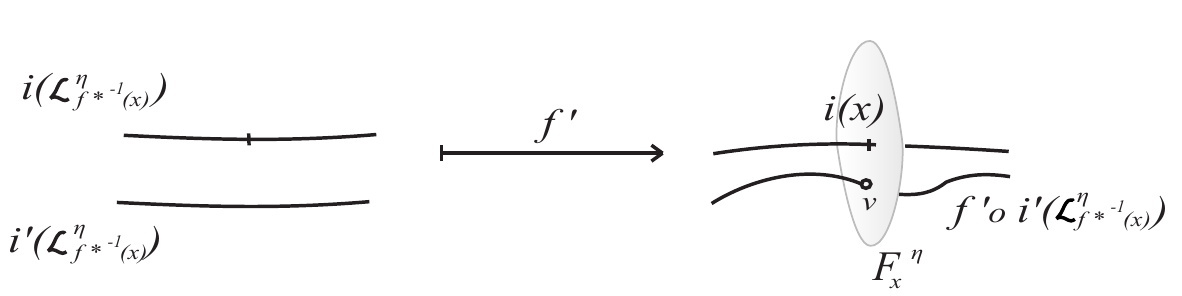}
        
    \caption{Definition of $S^0$.}
    \label{defS}
\end{figure}

Let $S^0(f',i')(x)$ be the point $v$.


Such a map $S^0$ satisfies conclusions 1 and \ref{preconc4'} of Lemma \ref{lem6'}.
Let us show that $S^0$ takes continuously its values in the set of morphisms from the lamination $\mathcal L$ restricted to $V$ into $M$.   

A small open neighborhood $U$ of $x$ has the closure of its backward image by $f^*$ included in a distinguish subset of $\mathcal L$. Let $\mathcal L_y$ be some of the plaques provided by a chart of this distinguish subset, containing ${y\in f^{*^{-1}}(U)}$. 

We notice that the submanifold $f'\circ i'( \mathcal L_{y})$ depends $C^r$-continuously on ${y\in f^{*^{-1}}(U)}$, $i'$ $C^r$-close to $i$, and $f'$ $C^r$-close to $f$.
  
%

For $\eta>0$ small enough, the manifolds $( F_{x''}^{\eta})_{x''\in \mathcal L_{x'}^{\eta}}$
 are the leaves of a $C^r$-foliation, which depends 
  $C^r$-continuously on $x'\in V$:  the foliation associated to a point $x'_1$ close to $x'$ is sent to the one of $x'$ by a $C^r$-diffeomorphism of $M$ close to the identity. 
 
For $\eta>0$ sufficiently small and then $f'$ and $i'$ $C^r$-close enough to $f$ and $i$, each submanifold $ F_{x''}^{\eta}$
 intersects transversally at a unique point  the submanifold $f'\big(i'(\mathcal L_{f^{*-1}(x')})\big)$, where $x'$ belongs to $U$ and $x''$ to $\mathcal L_{x'}^{\eta}$. 
As we know $S^0(f',i')(x'')$ is this intersection point.
 
In other words, $S^0(f',i')_{|\mathcal L_{x'}^{\eta}}$ is the composition of $i$ with the holonomy along the $C^r$-foliation $(F_{x''}^{\eta})_{x''\in \mathcal L_{x'}^{\eta}}$, from $i(\mathcal L_{x'}^{\eta})$ to the transverse section $f'\circ i'\big(\mathcal L_{f^{*-1}(x')}\big)$, 
 
 Thus, the map $S^0(f',i')$ is of class $C^r$ along the $\mathcal L$-plaques contained in $U$.

 As these manifolds vary $C^r$-continuously with $x'\in U$,  the map $S^0(f',i')_{|U}$ is a
 $\mathcal L_{|U}$-morphism into $M$.
 
 Also by transversality, the map $S^0_{|U}$ takes its values in the subset of immersions. 
 
  As these foliations and manifolds also depend $C^r$-continuously on $f'$ and $i'$, there are neighborhood $V_f$ of $f\in End^r(M)$ and $V_i^0$ of $i_{|L'}\in Im^r(\mathcal L_{|L'},M)$ such that the map
  \[S^0_{|U}\;:\;(f',i')\in V_f\times V_i^0 \mapsto S^0(f',i')_{|U}\]
is well defined and continuous into $Im^r(\mathcal L_{|U}, M)$.

  As $cl(L')$ is compact, it has a finite open cover of  by such open subsets $U$ 
  on which the restriction of $S^0$ satisfies the above regularity property. 
  By taking $V_i^0$ and $V_f$ small enough to be convenient for all the subsets of this finite cover, we get the continuity of the following map:
  \[S^0\;:\;(f',i')\in V_f\times V_i^0 \mapsto S^0(f',i')\in Im^r(\mathcal L_{| V}, M).\]
where $ V$ is the union of the open cover of $f^*(K)$.

\end{proof}
By conclusion 2 of Lemma \ref{lem6'}, for any $(f',i')\in V_f\times V_i^0$, we shall identify $S^0(f',i')$ with the $C^r$ section of the bundle: $(F,\mathcal F)\rightarrow (L,\mathcal L)$ restricted to $V$.

For a function $\rho \in Mor^r (\mathcal L,[0,1])$, with compact support included in $V$ and equal to $1$ on a neighborhood $V'$ of $L'$, we define by using the above identification:

\[S\; :\; V_f\times V_i^0\rightarrow Im^r(\mathcal L,M)\]
\[(f',i')\mapsto S_{f'}(i'):= \left\{\begin{array}{cl} \rho(x)\cdot S^0(f',i')(x)& \mathrm{if}\; x\in V\\
i(x)& \mathrm{else}\end{array}\right.\]
For $f'\in V_f$, we want to show that $S_{f'}$ has a unique fixed point $i(f')$.
 
The space $\Gamma$ of $C^r$-sections of $(F,\mathcal F)\rightarrow (L,\mathcal L)$ with support in $V$ is a Banach space. Unfortunately $S$ is not a contraction for such a space. Nevertheless it is a contraction for the $C^0$-uniform norm $d_{C^0}$ of $\Gamma$ induced by the norm of $F$.

\subsection{The $C^0$-contraction}
By definition of $S$, the $C^0$-contraction of $S^0$ implies the one of $S$.

Let $x\in f^{*^{-1}}(V)\subset D$. We shall identify both foliations $(F_{x'}^\eta)_{x'\in \mathcal L_{x}^\eta}$ and $(F_{y'}^\eta)_{y'\in \mathcal L_{f^{*}(x)}^\eta}$ to $\mathbb R^{n-d}\times \mathbb R^d$.
Let $p_1\;:\; \mathbb R^{n-d}\times \mathbb R^d\rightarrow \mathbb R^{n-d}$ and $p_2\;:\; \mathbb R^{n-d}\times \mathbb R^d\rightarrow \mathbb R^d$ be the canonical projections.

For $f'\in V_f$, $i'\in V_i^0$ and $x\in f^{*-1}(V)$, let $y:= p_2\circ  f'\circ i'(x)$. The point $z=S_{f'}^0(i')(y)$ is a solution of the equation:
\[f'\circ i'(x)=z.\]
 Thus, the point $z=S_{f'}^0(i')(y)$ is also a solution of:
\[p_1\circ f'\circ i'(x)=p_1(z).\]
As $S^0(f',i')(y)$ belongs to $F_y^\eta$, this point is identified to $(v,y)\in \mathbb R^{n-d}\times \mathbb R^d$.

As $\mathcal L_{x}^\eta$ is precompact in $L$, the restriction $p_1\circ i'_{|\mathcal L_x^\eta}$ is identified to 
a vector of the Banach space
of $C^1$-bounded maps $C^1_b(\mathcal L^\eta_{x},\mathbb R^{n-d})$. 
 
Let us apply the implicit function theorem with the following $C^1$-map from a neighborhood of $(0,0)$:
\[\Psi_{x,f'}\; : \; C^1_b(\mathcal L_{x}^\eta,\mathbb R^{n-d})\times \mathbb R^{n-d}\rightarrow \mathbb R^{n-d}\]
\[(j,v)\mapsto p_1\circ f'\circ j(x)-v.\]

 We may assume that $(\Psi_{x,f'})_{x\in f^{*-1}(V),f'\in V_f}$ is locally a continuous family of $C^1$-maps, since the charts of $\mathcal L$ provide continuous identifications of the $\eta$-plaques centered at points of $f^{-1}(V)$. 
 
 Also, $\Psi_{x,f}(0,0)$ is zero for any $x\in f^{*-1}(V)$ and $\partial_{v} \Psi_{x,f'}=-id$ is invertible. 
 Thus, by precompactness
 of $f^{*-1}(V)$, the implicit function theorem implies that $i'\in \Gamma \mapsto S^0_{f'}(i')(y)$ is of class $C^1$, for any $y\in V$ and $f'$ close to $f$. Moreover, $\Psi_{x,f'}$ depends $C^1$-continuously on $x$ and $f'$.
The differential
 \[\partial_i (S^0(f,i)(y))(j)= -\partial_v\Psi(x,f)^{-1}\circ T\Psi(x,f)(j)=p_1\circ Tf\circ j(x)\]
is contracting by normal contraction. Thus, $\partial_i(S^0(f',i')(y))$ is also $\lambda$-contracting for $f'$ close to $f$ and $i'\in \Gamma$ close to $i$.

Thus, by the mean values theorem, $S^0(f',\cdot)$ and so $S_{f'}$ are $\lambda$-contracting on a fixed neighborhood of $i$, for every $f'\in V_f$.

Also $S_f$ sends the intersection of $V_i^0$ with the $\epsilon$-$C^0$-ball centered at $i$ into a $\lambda\cdot\epsilon$-ball, for any small $\epsilon>0$.

By continuity of $f'\mapsto S_{f'}$,  by restricting $V_f$, for every $f'\in V_f$, 
$S_{f'}$ sends $i$ into $B_{C^0}(i, (1-\lambda)\epsilon)$ and so $cl(B_{C^0}(i,\epsilon))\cap V_i^0$ into $B_{C^0}(i,\epsilon)\cap \Gamma$.

\subsection{The $1$-jet space}\label{r=1c}

Contrarily to what happened for the $C^0$-topology, the map $S_f$ is not necessarily contracting in the $C^1$-topology, even when $L$ is compact (and so $S=S^0$). However, we are going to show that the forward action of $Tf$ on the Grassmannian of $d$-planes of $T\mathcal F$ is contracting, at the neighborhood of the distribution induced by $T\mathcal L$. Let us describe how $Tf$ acts on $(F,\mathcal F)$. 

 By compactness of $cl(V)$, there exists $\epsilon>0$ such that the restriction of $I$ to any $\epsilon$-plaque $\mathcal F_{x}^\epsilon$ of $\mathcal F$, centered at $x\in cl(V)$, is a diffeomorphism onto its image which is open in $M$. We denote by $I^{-1}_x$ its inverse. Thus, by restricting $V_f$, on a neighborhood $U^*$ of the zero section of $F_{|f^{*-1}(V)}$, we can define, for every $f'\in V_f$:
 \[\hat f' \; :\; z\in U^*\mapsto I^{-1}_{f^*\circ \pi(z)}\circ f'\circ I(z),\]
which is a $C^r$-morphism from $\mathcal F_{|U^*}$ into $\mathcal F$.

Let us define a normed vector bundle structure on a neighborhood of $T\mathcal L$ in the Grassmannian of $d$-planes of $T\mathcal F$. Let $\chi_0$ be a continuous, horizontal distribution of $d$-planes of $(F,\mathcal F)\rightarrow (L,\mathcal L)$. By horizontal we mean that for every $y\in F$, the $d$-plane $\chi_0(y)$  of $T_y\mathcal F$ is sent onto $T_{|\pi(y)}\mathcal L$ by $T_{y}\pi$.

 Let $\Psi\; :\; (y,t)\in F\times \mathbb R\mapsto t\cdot y\in F$. The map $\Psi$ is a $C^r$-morphism of $(F,\mathcal F)\times \mathbb R$ onto $(F,\mathcal F)$. Let $\chi$ be the continuous, horizontal distribution defined by:
\[\chi(y)=\left\{\begin{array}{cl} \partial_y \Psi_{(y,\|y\|)}\Big(\chi\big(\frac{y}{\|y\|}\big)\Big)& \mathrm{if}\; y\not=0,\\
T_y \mathcal L& \mathrm{if} \; y=0.\end{array}\right.\]
For all $x\in L$ and $y\in F_x$, the vector space $T_x\mathcal L^*\otimes F_x$ of linear maps from $T_x\mathcal L$ into $F_x$ is isomorphic to the space $\chi(y)^*\otimes F_x$, by the isomorphism which associates to $l\in T_x\mathcal L^*\otimes F_x$ the map $l\circ T\pi_{|\chi(y)}$.

As $F_x$ is canonically isomorphic to the subspace of $T_y\mathcal F$ tangent to the fiber $F_x$, via the graph map we shall identify $\chi(y)^*\otimes F_x$ (and so $T_x\mathcal L^*\otimes F_x$) to a neighborhood of $\chi(y)$ in the Grassmannian of $d$-planes of $T_y\mathcal F$. We denote by $P^1_y$ such a vector space that we endow with the norm subordinate to those of $T_x\mathcal L$ and $F_x$. 
We denote by $P^1$ the vector bundle over $F$ whose fiber at $y\in F$ is $P^1_y$.\\

 As $f^*$ is an immersion and $f$ preserves the immersion of $(L,\mathcal L)$, for $f'$ close enough to $f$ and for
 all small $z\in (f^*\circ \pi)^{-1}(V)$, for any $d$-plane $l\in P^1_z$, the image by $T\hat f'$ is a small $d$-plane $l'\in P^1_{\hat f'(z)}$.
  
Let us show the following lemma:
\begin{lemm}\label{contracP1}
For all small $\epsilon>0$ and then $V_f$ small enough, for all $f'\in V_f$, $z\in (f^*\circ \pi)^{-1}(V)$ and $l\in P^1_{z}$ both with norm not greater than $\epsilon>0$, the norm of $\phi_{f'z}(l):=l'\in P^1_{\hat f(z)}$ is less than $\epsilon$. Moreover, the map $\phi_{\hat f'z}$ is $\lambda$-contracting.\end{lemm}
\begin{proof}
Let $\pi_v$ be the projection of $T\mathcal F$ onto $F$ parallelly to $\chi$. Let $\pi_h$ be the projection of $T\mathcal F$ onto $\chi$ parallelly to $F$. Let $Tf'_h:=\pi_h\circ T\hat f'$ and $Tf'_v:=\pi_v\circ T\hat f'$.

For any vector $e\in \chi(z)$, the point $(e,l(e))$ is sent by $T_z \hat f'$ onto $\big(Tf'_h(e,l(e)),Tf'_v(e,l(e))\big)$.

Let $e':=  Tf'_h(e,l(e))$. By definition of $l'$, the point $(e,l(e))$ is sent by $T_x\hat f'$ to $(e',l'(e'))$.

As $f^*$ is an immersion, by restricting $V_f$ and $\epsilon$, the map $e\mapsto Tf_h'(e,l(e))$ is invertible. 
Thus $e=(Tf'_h(\cdot,l(\cdot)))^{-1}(e')$.
 
Therefore, the map $l'$ is $e'\mapsto Tf'_v(\cdot ,l(\cdot))\circ (Tf'_h(\cdot ,l(\cdot))^{-1} (e')$.

Hence, the expression of $l'=\phi_{f'z}(l)$ depends algebraically  on $l$, and the coefficients of this algebraic expression depends continuously on $f'$ or $z$, with respect to some trivializations.

When $f'$ is equal to $f$ and $z$ belongs moreover to the zero section, the map 
\[\phi_{fz} \;:\;l'\mapsto T_zf_v\circ l\circ \big(T_zf_h(\cdot ,l(\cdot))\big)^{-1}\]
is $\lambda$-contracting for $l$ small, by normal contraction, since $T_zf_h(\cdot ,l(\cdot))$ is close to $T_zf^*$.

Thus, for $\epsilon$ small enough, for all $z\in (f^*\circ \pi)^{-1}(V)$ and $l \in P^1_{z}$ both with norm less than $\epsilon$, the tangent map of $\phi_{fz}$ has a norm less than $\lambda$.

Therefore, for $V_f$ small enough, the tangent map $T\phi_{f'z}$ has a norm less than $\lambda$ and hence $\phi_{f'z}$ is a $\lambda$-contraction on $cl\big({B}_{P_{z}^1}(0,\epsilon)\big)$. 

As, for $z$ in the zero section, the map $\phi_{fz}$ vanishes at $0$, for $\epsilon$ and $V_f$ small enough, the norm $\phi_{f'z}(0)$ is less than $(1-\lambda)\cdot \epsilon$.

Consequently, by $\lambda$-contraction, the closed $\epsilon$-ball centered at the $0$-section is sent by $\phi_{f'}$ into the $\epsilon$-ball centered at $0$, for all $z\in (f^* \circ \pi)^{-1}(V)$ with norm less than $\epsilon$ and $f'\in V_f$.
 \end{proof}
\begin{rema}\label{futurgrass} If $r\ge 1$, we notice that the action of $\hat f$ on $P^1$ is contracting and $(r-1)$-times more than $f^*$.\end{rema}

For any $C^1$-section $i'$ of $(F,\mathcal F)\rightarrow (L,\mathcal L)$
and for every $x\in L$, the tangent space to the image of $i'$ at $i'(x)$ is an element $\nabla i'(x)$ of $P^1_{i'(x)}\approx T_x\mathcal L^*\oplus F_x$.

By conclusion 2 of Lemma \ref{lem6'}, for $f'\in V_f$, for $j\in V_i^0$, for $x_0\in (\pi \circ \hat f'\circ j)^{-1}(V)$, the $d$-plane $\nabla j(x_0)$ is sent by $T_x\hat {f'}$ to  $\nabla S^0(f',j)(y_0)$, with $x:= j(x_0)$ and $y_0:=\pi\circ \hat f'(x)$. Thus $\nabla S^0(f',j)(y_0)$ is equal to $\phi_{f'x}(\nabla j(y_0))$.

 The interest of such a distribution $\chi$ is that any morphism $p \in Mor^1 (\mathcal L,\mathbb R)$ and $i'\in \Gamma$ satisfy: 
\begin{equation}\label{nabla}\nabla(p\cdot i')(x)=Tp(x)\cdot i'(x)+p(x)\cdot \nabla i'(x),\quad \forall x\in L\end{equation}    
\begin{proof}[Proof of Equation (\ref{nabla})]
 Let $\pi_v$ be the projection of $T\mathcal F$ onto $F$ parallelly to $\chi$. The linear morphism $\pi_v$ commutes with the partial
derivative $\partial_1\Psi$ with respect to the first variable since $\pi_v\circ \partial_1 \Psi$ and $\partial_1 \Psi\circ \pi_v$ have the same kernel $\chi$ and are equal on the complement $F$ in $T\mathcal F$.

 Thus we have:
\[\nabla(p i')=\pi_v\circ T( \Psi(i', p))= \pi_v\big( \partial_1\Psi_{(i', p)}\circ Ti'+\partial_2\Psi_{(i', p)}\circ T p\big)=
\partial_1 \Psi_{(i', p)}\circ \pi_v\circ Ti'+ \pi_v\circ \partial_2 \Psi_{(i', p)}\circ T p.
\]
Also we have: $\partial_1 \Psi_{(y,t)}(u)=u\cdot t$ and $\partial_2 \Psi_{(y,t)}(t')=t'y$, for every $y\in   F_{\pi(y)}$. Consequently:

\[\nabla(p i')=p\cdot \pi_v\circ Ti'+T\rho\cdot i'=p \cdot \nabla i'+Tp\cdot i'.\]
\end{proof}

Thus by Equation (\ref{nabla}), we have for any $x_0\in (\pi \circ \hat f'\circ j)^{-1}(V)$:
\[\nabla S_{f'}(j)(y_0)=\rho(y_0)\cdot \phi_{f'x} (\nabla j(x_0)) +T\rho(y_0)\cdot \hat f'(x),\quad \mathrm{with}\; x:=j(x_0)\; \mathrm{and}\; y_0:=\pi\circ \hat {f}'(x).\]

Since the support of $\rho$ is in $V$, for every $y_0\in L$:
\[\nabla S_{f}(j)(y_0)=\phi'_{fx}(\nabla j(y_0)),\]
 
with $\phi'_{fx}\; :\; l\in P^1_{x}\mapsto \left\{\begin{array}{cl} 
0& \mathrm{if}\;\pi(y_0)\notin V\\
\big( \rho (y_0)\cdot \phi_{f'x} (l) + T_{y_0} \rho \cdot x\big)\in P^1_{\hat f'(x)}&\mathrm{if}\;  \pi(y_0)\in V.\end{array}\right.$.  

We notice that $\phi'_{f' x}$ is at least as much contracting as $\phi_{f'x}$ and sends the $\epsilon$-ball of $P^1_{x}$ centered at 0 into the $\epsilon$-ball of $P^1_{\hat f'(x)}$ centered at $0$, for $x$ small and $f'$ close to $f$. Also $(\phi'_{f'x})_x$ is a continuous family of maps.

\subsection{Proof of Theorem \ref{Maincontract} when $r=1$}\label{Proof of theorem}
Let us begin by proving the existence of a closed neighborhood $V_{i}$ of $i\in \Gamma$ sent by $S_{f'}$ into itself, for any $f'\in End^r(M)$ close to $f$.

For $\epsilon, \epsilon' >0$ small enough, the following set is included in $V_i^0$: 

\[V_i:= 
\Big\{i'\in \Gamma; \quad d_{C^0}(i,i')\le \epsilon'\;\mathrm{and}\; \; \|\nabla i'(x)\|_{P^1}\le \epsilon,\; \forall x\in L\Big\}\]
where $\|\cdot \|$ refers to the $P^1$-norm.

 Therefore, by the two last steps, for $\epsilon$ then $\epsilon'$ and finally $V_f$ small enough, $S_{f'}$ sends $V_i$ into itself, for every $f'\in V_f$. 

We showed that $S$ is contracting for the $C^0$-topology. Hence, for every $f'$ in $v_f$, the sequence $(S_{f'}^n(i))_{n\ge 0}$ converges to a Lipschitz section $i(f')$. If this section is continuously differentiable, then the morphism 
\[f'^*:= \; x\in L'\mapsto \pi\circ \hat f'\circ i'(x)\]
is also of class $C^1$. Moreover Conclusions 1 and 2 of Theorem \ref{Maincontract} are satisfied. 

Let us show that $i(f')$ is of class $C^1$. 

We remind that for every $x\in V$ and every $i'\in V_i$, there exists $y\in   L$ such that:
\[ \rho(x)\cdot \hat f'\circ i'(y)= S_f(i')(x).\]

Consequently, for every $x\in V$, there exists $y$ such that: \[\rho(x)\circ \hat f\circ i(f')(y)= i(f')(x).\]

Thus there are two possibilities for each point $x\in V$:
\begin{enumerate}
\item there exists an infinite sequence $(x_n)_{n\le 0}\in V^\mathbb N$ such that:
\begin{equation}\label{the very last one}x_0=x\; \mathrm{and} \; \rho(x_{n+1})\cdot \hat f'\circ i'(x_n)= S_f(i')(x_{n+1})\end{equation}
\item there is not such an infinite sequence.
\end{enumerate}

In the second case, let $(x_n)_{n=0}^{-N}\in V^N$ be a maximal chain satisfying (\ref{the very last one}).

Then we notice that 
\[i(f')(x)= S_{f'}^{N+1} (i)(x).\]

Since $\rho$ has compact support in $V$, the above equality holds on a neighborhood of $x$. This implies that $i(f')$ is of class $C^1$ on a neighborhood of $x$. 

The first case  is more delicate. Let $(x_n)_{n\le 0}$ be as in the definition. By Lemma \ref{contracP1}, the intersection   
\[\bigcap_{n\ge 0} \phi'^{n}_{f'}(B_{P^1_{i(f')(x_n)}} (0,\epsilon))\]
is decreasing and consists of a single point  $l(x)$. Also for $x'$ close to $x$, the length of the preorbit of $x'$ in $V$ is  large.

Let $u$ be a tangent vector at $i(f')(x)$ to $i(f')(\mathcal L_x^\eta)$. This means that there exists a sequence $(y_n)_n\in \big(i(f')(\mathcal L_x^\eta)\setminus \{y\}\big)^\mathbb N$ such that:
\[y_n\stackrel{n\rightarrow \infty}{\longrightarrow} y\quad \mathrm{and}\quad y_n= y+ud(y_n,y)+o(d(y_n,y)),\quad \mathrm{
with} \; y:= i(f')(x).\]

Also, for every  $k\le 0$ and then $n$ sufficiently large, there exists $(y_n^j)_{j=0}^k\in V^{-k}$ such that $y_n^0=y$ and:
\[ \rho \circ \pi (y_n^{j+1} ) \circ \hat f'(y_n^j) =y_n^{j+1}, \quad \forall j<0\]

Moreover the sequence $(y_n^{j})_n$ converges to $y^j:=i(f')(x_j)$ and there exists $u_j$ such that: \[y_n^j=y_j+u_j d(y_n^j, y_j)+o(d(y_n^j, y_j)).\]

Consequently, $u_j$ is tangent to $i(f')(\mathcal L_{x_j}^\eta)$. Also  $u_j$ belongs to a $d$-plane of $B_{P^1_{y_j}}(0,\epsilon)$. Therefore $u_0$ is a $\epsilon\cdot \lambda^{-j}$-close to $l$. When $j$ approaches infinity, we get that the tangent space of $i(f')(x)$ is $l$.

Thus we showed that at any point $x\in   L$, the tangent space $\nabla i(f')(x)$ of $i(f')(\mathcal L_x^\eta)$ at $x$ is a $d-$plane of $P^1_{i(f')(x)}(0,\epsilon)$.

To show that $i(f')$ is of class $C^1$, it only remains to show that such a distribution of $d$-planes is continuous at every point $x$ with infinite preorbit $(x_n)_{n\le 0}$.

Let $\delta>0$. Let $N\ge 0$ such that $\lambda^N \epsilon\le \delta/3$. Let $x'$ be sufficiently close to $x$ such that it has a preorbit $(x_{n}')_{n=0}^{-N}\in V^\mathbb N$ well defined. 

By continuity of $\phi_{f'}'$, for $x'$ sufficiently close to $x$, the chain $(x_n ')_{n=0}^{-N}$ is sufficiently close to $(x_n)_{n=0}^{-N}$ such that $l_1:= \phi'_{f'y_{-1}}\circ \cdots \circ \phi'_{f'y_{-N}}(0)$ and $l'_1:= \phi'_{f'y'_{-1}}\circ \cdots \circ \phi'_{f'y'_{-N}}(0)$ are $\delta/3$-close, with $(y'_n)_n:= i(f')(x'_n)$.

Also by contraction of $\phi'_{f'}$, 
\[d(l_1,\nabla i(f')(x))\le \frac{\delta}{3}\quad \mathrm{and} \quad d(l'_1,\nabla i(f')(x'))\le \frac{\delta}{3}.\]

Therefore the distance between $\nabla i(f')(x)$ and $\nabla i(f')(x')$ is less than $\delta$. This finishes the proof of the regularity of $i(f')$.

Let us prove conclusion $(3')$ of Theorem \ref{Maincontract}.

Let $(y_j)_{j\le 0}\in M^{\mathbb Z^-}$ be a $f'$-preorbit which is $\epsilon$-close to the image by $i$ of a $\epsilon$-pseudo-orbit $(x_j)_{j\le 0}\in L'^{\mathbb Z^-}$ of $f^*$ respecting the plaques. 
Let $y_j':= (I_{x_j})^{-1}(y_j)$ and $z_j:= \pi(y'_j)$.

We want to show that $i(f')(z_j)$ is equal to $y_j$, for every $j\le 0$.

Let $n< 0$ be such that $i(f')(z_n)$ and $y_n$ are at distance in $F_{z_n}$ at least:
\begin{equation}\label{eq pour conclusion 3'} \frac{\lambda+1}{2}\sup_{j< 0}d_{F_{z_j}}\big(i(f')(z_j),y_j\big).\end{equation}

For $\epsilon>0$ small, the sequence $(z_j)_{j< 0}$ is close to $L'$. Also we can construct a section $i'\in V_i$ equal to $y_j$ at $z_j$, for each $j<0$, and such that the $C^0$-distance between $i'$ and $i(f')$ is $\sup_{j< 0}d_{F_{z_j}}\big(i(f')(z_j),y_j\big)$.

We remind that on $V'\subset V$ the maps $S$ and $S^0$ are equal, since $\rho$ is there equal to 1.
 
By contraction of $S^0$, the $C^0$-distance between $i(f')_{|V'}$ and $S^0(f',i')_{|V'}$ is less than $\lambda\cdot d(i(f')(z_n),y_n)$.

By conclusion $(2)$ of Lemma \ref{lem6'}, the immersion $S^0_{f'}(i')$ sends $(z_j)_{j\le 0}$ to $(y_j)_{j\le 0}$. By Equation (\ref{eq pour conclusion 3'}):
\[\sup_{j<0} d_{F_{z_j}}\big(i(f')(z_j),y_j\big)\ge d_{C^0} \big(i(f'), i'\big) \ge \frac{1}{\lambda} d_{C^0} \big(i(f'), S_{f'}(i')\big)\ge \frac{\lambda+1}{2\lambda}
\sup_{j<0} d_{F_{z_j}}\big(i(f')(z_j),y_j\big).
\]

Therefore for every $j< 0$, the point $z_j$ is sent by $i(f')$ to $z_j$. By commutativity of the diagram, the point $z_0$ is sent by $i(f')$ to $z_0$.

\subsection{Proof of Theorem \ref{Maincontract} when $r>1$}
Instead of regarding the immersion $i$ of $\mathcal L$ into $M$, we can consider the following immersion into the Grassmannian bundle $G_r(d,TM)$ of $d$-planes of $TM$:
 \[ i_1\;:\; x\in L\mapsto \big(i(x),Ti(T_x\mathcal L)\big).\]

Since $f^*$ is an immersion, $Tf$ acts canonically on the Grassmannian $G$ of $d$-planes of $TM$, on a neighborhood of $\{Ti(T_x\mathcal L), \; x\in L\}$. As $f$ is $1$-normally contracting, $Tf$ sends into itself a nice such a neighborhood.

We define by induction the $l^{th}$-Grassmannian $G^l$: $G^1$ is $G$ and $G^{l+1}$ is the Grassmannian of $d$-planes of $TG^l$. By the same way, we define the immersion $i_l$ of $(L,\mathcal L)$ into $G^l$. We notice that $Tf$ acts canonically on $G^l$. We denote by $f_l$ its action.

Let us describe the constructions $(G^j)_j$, $(i_j)_j$ and $(f_j)_j$. 

\[\mathrm{Let}\; G^0:= M\;  \mathrm{and}\; G^{j+1}:= \{l \in C^1(\mathbb R^d, G^j):\; T_0 l\; \mathrm{is\; injective}\}/\sim,\]
with $l\sim l '$ if there exists a diffeomorphism $\phi$ of $\mathbb R^d$ fixing the origin such that $l'(x)= l\circ \phi(x)+ o(x)$. 
We notice that we have a bundle $\pi_{j+1,j}:\; G^{j+1}\rightarrow G^j$ and so $\pi_{j}:= \pi_{j,j-1}\circ \dots \pi_{1,0}:\; G^j\rightarrow M$. 

For each $i'\in V_i$,  we define the following immersion of $(L,\mathcal L)$ into $G^j$: $i_0':= i'$ and $i'_{j+1}:= x\in (L,\mathcal L)\mapsto [i'_{j|\mathcal L_x^\eta}]$, by identifying $\mathcal L_x^\eta$ to $\mathbb R^d$ and $x$ to $0$.
We notice that $i'_j$ is well defined for $0<j\le r$ since $\pi_{j-1} \circ i'_{j-1}=i'$ is an immersion and so $i'$ is an immersion.

Similarly the action $ f'_j$ of $f'\in V_f$ on $G^j $ is defined inductively: $f'_0$ is $f'$ and $f'_{j+1}([l])$ is equal to $[f'_j\circ l]$.
This action is well defined on a neighborhood of the image of $i_j$ since $\pi_j\circ  f'_j\circ i_j= f\circ i$ and so $f'_j\circ i_j$ is well an immersion.  

By Remark (\ref{futurgrass}), $f_1$ $r-1$-contracts the immersion $i_1$ of $(L, \mathcal L)$ over $f^*$. Thus, by induction, the action of $ f'_l$ on the $ l^{th}$-Grassmannian $G^l$ of $TM$ is contracting at a neighborhood of $i_l(L)$,  for every $l\in \{1,\dots ,r\}$.    

As before we shall use the formalism of connexions since the lamination $(L,\mathcal L)$ is not compact, and so we are going to study the consequences  of the multiplication by $\rho$ of after the action of $T^j  f'$ on $G^j$.

In a similar way to which it is done in section \ref{r=1c}, for $i'\in \Gamma$ and $l\in \{1,\dots,r\}$, let $\nabla^l i'$ be the map $p_v\circ T^li'$,  where $\pi_v$ is the projection of $T\mathcal F$ onto $F$ parallelly to $\chi$, we notice that $T^l i'$ is the $l^{th}$-differential of $i'$. 
Thus $\nabla^l i'$ is a vector of the space $L_{sym}^l(T\mathcal L,F)$ of $l$-linear symmetric morphisms from the bundle $(T\mathcal L)^l$ to $F$ over $L$. 

Let us show the identification of $i'_j$ with $\nabla^j i'(x)$, by induction. For $j\le 1$ this is already done.  

We remind that $i'_{j+1}(x)=[i_{j|\mathcal L_x^\eta}]$, for $x\in L$. By induction, we have $i'_{j+1}(x)=[p_v \circ T^j  i'_{|\mathcal L_x^\eta}]$ and so $i'_{j+1}(x)= Tp_v\circ T_x T^j i'$. As $p_v$ is the tangent map of $T_{i'(x)}\mathcal F$, we have:
\[i'_{j+1}(x)= p_v \circ T^{j+1}_x i'.\]

We notice that this identification is coherent with the topology of $G^j$ and $Im^j (\mathcal L,M)$.

By the above discussion, we can define an  endomorphism $\phi^l_{f'}$ from the $\epsilon$-neighborhood of the zero section of $L_{sym}^l(T\mathcal L,F)$ into itself, contracting each fiber and such that:
\[\forall x\in V,\quad \nabla^l S^0(f',i')(x):= \phi^l_{f'y}(\nabla^l i'(y)), \quad \mathrm{with}\; \hat f'(i'(y))=x.\]

An induction proves similarly as for Equation (\ref{nabla}), the following Leibniz formula: 

\[\nabla^l S_{f'}(i'):=\sum_{k=0}^l C_l^k (T^{l-k} \rho)  \nabla^k S^0(f',i').\]

Thus $(\nabla^l S_{f'}(i'))_{l=1}^r$ is equal to a continuous, contracting function of $(\nabla^k i')_{k=1}^l$  and we can proceed as in section \ref{Proof of theorem} to conclude the proof. 

\section{Proof of differentiable persistences of normally expanded laminations}
In this section we prove Theorem \ref{Main} in the normally expanded case.

\subsection{The setting}
We fix $r\ge 1$. Let us denote by $K$ the closure of $L'$. 

Let $(F,\mathcal F,I,\pi)$ be a $C^r$-tubular neighborhood of the immersed lamination $(L,\mathcal L)$ in $M$.
We identify $L$ to the zero section of $F$. We endow $i^*TM$ with a norm satisfying the normal expansion inequality. We endow $F$ with the norm associating to a vector $v\in F$ the one of the projection of $T_0I(v)$ into $E^u$ parallelly to $T\mathcal L$ in $i^*TM$.

Let $d$ be the dimension of the leaves of $\mathcal L$. 

We will always suppose $\eta>0$ small enough such that for every $x$ in a neighborhood of $K\cup f^*(K)$:\begin{itemize}
\item[-]  the ball $B_{F_x}(0,\eta)$ centered at $0$ and with radius $\eta$ is sent by $I$ to a submanifold $F_x^\eta$,
\item[-]  the union $\mathcal L_x^\eta$ (resp. $\mathcal F_x^\eta$) of plaques of $\mathcal L$ (resp. $\mathcal F$) containing $x$ and with diameter less than $\eta$, is a precompact plaque. \end{itemize}



\subsection{The algorithm}  
The idea of the proof is to use the following lemma which constructs the preimages of perturbations of the immersed lamination  $\mathcal L$.

\begin{lemm}\label{lem6}
For any small $\eta>0$, there exists a neighborhood $V$ of $K$ in $L$, a neighborhood $V_i^0$ of $i\in Mor^r(\mathcal L,M)$, a neighborhood $V_f$ of $f\in End^r(M)$ and a continuous map
\[S^0\;:\; V_f\times V_i^0\rightarrow Im^r(\mathcal L_{|V},M)\]
satisfying:
\begin{enumerate}
\item the morphism $S^0(f,i)$ is equal to $i_{|V}$,
\item \label{preconc4} for all $x\in L'$, $i'\in V_i^0$ and $f'\in V_f$, the preimage of $i'(\mathcal L_{f^*(x)}^{\eta})$ by $f'$ intersects $ F_{x}^{\eta}$ at a unique point $S^0(f',i')(x)$, for all $f'\in V_f$ and $i'\in V_i^0$.
\end{enumerate}

\end{lemm}

\begin{proof}
By normal  expansion, we may suppose $\eta>0$ small enough and then $(i',f)$ close to $(i,f)$ such that, by normal expansion, the restriction of $f'$ to $F_{x}^{\eta}$ is a diffeomorphism onto its image, and this image intersects transversally at a unique point the image of the plaque $\mathcal L_{f^*(x)}^{\eta}$ by $i'$, for all $x$ in a neighborhood $V$ of $K$, for all $i'$ $C^r$-close to $i$ and for all $f'$ $C^r$-close to $f$.

Writing this intersection point in the form $f'(v)=i'(x')$, we define $S^0(f',i')(x):= v$. 
\begin{figure}[h]
    \centering
        \includegraphics{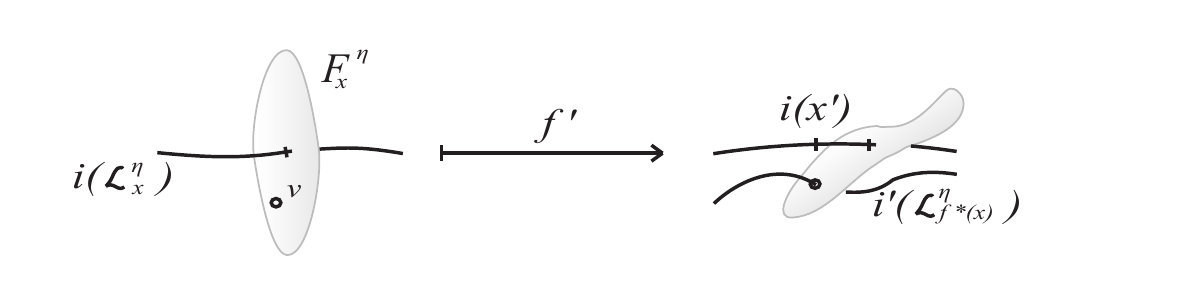}
        
    \caption{Definition of $S$.}
    \label{cexple}
\end{figure}

Such a map $S^0$ satisfies conclusions 1 and \ref{preconc4} of Lemma \ref{lem6}.
Let us show that $S^0$ takes continuously its values in the set of morphisms from the lamination $\mathcal L_{| V}$ into $M$.   

Let $x\in K$ and let $(U_y,\phi_y)\in \mathcal L$ be a chart of a neighborhood of $y:=f^*(x)$.
We may suppose that $\phi_y$ can be written in the form:
\[\phi_y\;: \; U_y\rightarrow \mathbb R^{d}\times T_y,\]
where $T_y$ is a locally compact metric space. Let $(u_y,t_y)$ be defined by $\phi_y(y)=:(u_y,t_y)$.

We remark that $\mathbb R^{d}$ is $C^r$-immersed by $\psi:= u\in \mathbb R^{d} \mapsto i\circ \phi^{-1}(u,t_y)$.

By normal expansion, the restriction of $f$ to a small neighborhood of $i(x)$ is transverse to the above immersed manifold, at $z:=f\circ i(x)$.
In other words,
\[Tf(T_{i(x)}M)+T\Psi(T_{u_y}\mathbb R^{d})=T_zM.\]

Thus, by transversality, there exist nice open neighborhoods $V_{u_y}$ of $u_y\in\mathbb R^d$
 and $V_{i(x)}$ of $i(x)\in M$ such that the preimage by $f_{|V_{i(x)}}$ of $\psi(V_{u_y})$ is a
 $C^r$-submanifold.
Moreover, such a submanifold depends continuously on $f$ and $\psi$, with respect to the $C^r$-topologies.

More precisely, there exist neighborhoods  $V_{u_y}$ of $u_y$,
$V_f$ of $f\in End^r(M)$, $V_\psi$ of $\psi\in C^r(\mathbb R^{d},M)$, and $V_{i(x)}$ of $i(x)$
such that for all $f'\in V_f$ and $\psi'\in V_\psi$, $f'_{|V_{i(x)}}$ is transverse to $\Psi'_{|V_{u_y}}$ and also the preimage by $f'_{|V_{i(x)}}$ of $\psi'(V_{u_y})$ is a manifold which depends continuously on $f'$ and $\psi'$, in the compact-open $C^r$-topologies.

There exist neighborhoods $V_{t_y}$ of $t_y$ in $T_y$ and $V_i^0$ of $i\in Im^r(\mathcal L, M)$, such that
\[\psi_{i',t}\;:\; u\in \mathbb R^d\mapsto i'\circ \phi(u,t)\]
belongs to $V_\psi$, for all $t\in V_{t_y}$ and $i'\in V_i^0$.

Thus, the preimage by every $f'\in V_f$, restricted to $V_{i(x)}$, of the plaque $\mathcal L_{t}:= \phi^{-1}_y(V_{u_y}\times \{t\})$,
immersed by $i'\in V_i^0$, depends $C^r$-continuously on $f'$, $i'$ and $t$.

For $\eta>0$ small enough, the manifolds $(F_{x''}^{\eta})_{x''\in \mathcal L_{{t}}^{\eta}}$
form the leaves of a $C^r$-foliation on a open subset of $M$, which depends continuously on $t\in   V_{t_y}$: the foliation associated to $t'$ close to $t$ is sent to the one of $t$ by a $C^r$-diffeomorphism of $M$ close to the identity. We may suppose $U_x$ and $\eta$ small enough in order that the closure of $\cup_{x'\in U_x}\mathcal L_{x'}^\eta$ can be sent by $f^*$ into $\phi_y^{-1}(V_{u_y}\times V_{t_y})$.
 
For all $\eta>0$ and then $V_i^0$ and $V_f$ small enough, each submanifold $F_{x''}^{\eta}$
 intersects transversally at a unique point  the submanifold $f'^{-1}_{|V_{i(x)}}\big(i'(\mathcal L_{t'})\big)$, where $t'$ is the second  coordinate of $\phi_y\circ f^*(x'')$ and $x''$ belongs to $\cup_{x'\in U_x} \mathcal L_{x'}^\eta$.
As we know $S^0(f',i')(x'')$ is this intersection point.
 
In other words, $S^0(f',i')_{|\mathcal L_{x'}^{\eta}}$ is the composition of $i$ with the holonomy along the $C^r$-foliation $(F_{x''}^{\eta})_{x''\in \mathcal L_{x'}^{\eta}}$, from $i(\mathcal L_{x'}^{\eta})$ to the transverse section $f'^{-1}_{|V_{i(x)}}(i'(\mathcal L_{t'}))$, where $t'$ is the second coordinate of $\phi_y\circ f^*(x')$, for every $x'\in U_x$.
 
 Thus, the map $S^0(f',i')$ is of class $C^r$ along the $\mathcal L$-plaques contained in $U_x$.
 As these manifolds vary $C^r$-continuously with $x'\in U_x$, the map $S^0(f',i')_{|U_x}$ is a
 $\mathcal L_{|U_x}$-morphism into $M$.
 
  As these foliations and manifolds depend $C^r$-continuously on $x'\in U_x$, $i'\in V_i^0$ and $f'\in V_f$, the map
  \[(f',i')\in V_f\times V_i^0 \mapsto S^0(f',i')_{|U_x}\]
  is continuous into $Mor^r(\mathcal L_{|U_x}, M)$.

  As, $K$ is compact, we get a finite open cover of $K$ by such open subsets $U_x$ 
  on which the restriction of $S^0$ satisfies the above regularity property. 
  By taking $V_i^0$ and $V_f$ small enough to be convenient for all the subsets of this finite cover, we get the continuity of the following continuous map:
  \[S^0\;:\;(f',i')\in V_f\times V_i^0 \mapsto S^0(f',i')\in Mor^r(\mathcal L_{| V}, M).\]
where $ V$ is the union of the finite open cover of $K$.

As $S^0$ takes its values in the set of the composition of the immersion $I$ with the section of $(F,\mathcal F)\rightarrow (L,\mathcal L)$, $S^0$ takes its values in the set of immersions from $\mathcal L_{| V}$ into $M$.

\end{proof}
By conclusion 2 of Lemma \ref{lem6'}, for any $f'\in V_f$ and $i'\in V_i^0$, we shall identify $S^0(f',i')$ with the $C^r$ section of the bundle: $(F,\mathcal F)\rightarrow (L,\mathcal L)$ restricted to $V$.

For a function $\rho\in Mor(\mathcal L,[0,1])$, with compact support included in $V$ and equal to $1$ on a neighborhood $V'$ of $L'$, we define by using the above identification:

\[S\; :\; V_f\times V_i^0\rightarrow Im^r(\mathcal L,M)\]
\[(f',i')\mapsto S_{f'}(i'):= \left\{\begin{array}{cl} \rho(x)\cdot S^0(f',i')(x)& \mathrm{if}\; x\in V\\
i(x)& \mathrm{else}\end{array}\right.\]
For $f'\in V_f$, we want to show that $S_{f'}$ has a unique fixed point $i(f')$.
 
The space $\Gamma$ of $C^r$ sections of $(F,\mathcal F)\rightarrow (L,\mathcal L)$ whose restriction to the complement of $V$ is equal to $0$ (that is $i$ in the identification) is a Banach space. Unfortunately $S$ is not a contraction for such a space. But it is a contraction when $\Gamma$ is endowed with the $C^0$-uniform norm $d_{C^0}$ induced by the norm of $F$.
                                                           
\subsection{The $C^0$-contraction}\label{C0contraction}
By definition of $S$, the $C^0$-contraction of $S$ follows from the $C^0$-contraction of $S^0$.
 
For $x\in V$, we shall identify the foliations $(F_{x'}^\eta)_{x'\in \mathcal L_{x}^\eta}$
and $(F_{y}^\eta)_{y\in \mathcal L_{f^*(x)}^\eta}$ to $\mathbb R^{n-d}\times \mathbb R^d$.

For $f'\in V_f$ and $i'\in V_i^0$, the point $y=S^0(f',i')(x)$ is a solution of the equation:
\[f'(y)=i'\circ p_2\circ f'(y),\]
where $p_2\;:\; \mathbb R^{n-d}\times \mathbb R^d\rightarrow \mathbb R^d$ is the canonical projection. Thus, the point $y=S(f',i')(x)$ is also a solution of:
\[p_1\circ f'(y)=p_1\circ i'\circ p_2\circ f'(y),\]
with $p_1\;:\; \mathbb R^{n-d}\times \mathbb R^d\rightarrow \mathbb R^{n-d}$ is the canonical projection.

As $S^0(f',i')(x)$ belongs to $F_x^\eta$, in these identifications, it is equal to some $(v,x)\in \mathbb R^{n-d}\times \mathbb R^d$.

As $\mathcal L_{f^*(x)}^\eta$ is precompact, in these identifications, $p_1\circ i'$ belongs to the Banach space of $C^1$-bounded maps $C^1_b(\mathcal L^\eta_{f^*(x)},\mathbb R^{n-d})$. 
 
Let us apply the implicit function theorem with the following $C^1$-map from a neighborhood of $0$:
\[\Psi_{x,f'}\; : \; C^1_b(\mathcal L_{f^*(x)}^\eta,\mathbb R^{n-d})\times \mathbb R^{n-d}\rightarrow \mathbb R^{n-d}\]
\[(j,v)\mapsto p_1\circ f'(y)-p_1\circ j\circ p_2\circ f'(y),\]
where $y=(x,v)$.

 We may assume that $(\Psi_{x,f'})_{x\in K,f'\in V_f}$ is locally a continuous family of $C^1$-maps, since the charts of $\mathcal L$ and the trivializations of $F$ locally  provide identifications of $(F_{x'}^\eta)_{x'\in \mathcal L_x^\eta}$ and $(F_{y}^\eta)_{y\in \mathcal L_{f^*(x)}^\eta}$ which vary continuously with $x\in V$. 
 
 Also, $\Psi_{x,f}(0,0)$ is zero for any $x\in V$, and $T\Psi_{x,f'}(0,0)(\partial v)=p_1\circ Tf'(\partial v)$ is invertible.
 
 Thus, by precompactness of $V$, the implicit function theorem implies that $i'\mapsto S^0(f',i')(x)$ is of class $C^1$, for any $x\in V$ and $f'$ close to $f$. Moreover, its differential depends continuously on $x$, $i'$ and $f'$.
 The differential:
 \[TS^0({f},i)(x)(j)= \partial_v\Psi_{(x,f)}(0,0)^{-1}\circ T\Psi_{(x,f)}(0,0)(j)=(p_2\circ Tf_{|TI(F_x)})^{-1}\circ j\big(f^*(x)\big)\]
is contracting by normal expansion. By continuity,  $TS_{f'}(j)(x)(\partial i)$ is also $\lambda$-contracting for $f'$ close to $f$ and $i'$ close to $i$.

Thus, by the mean values theorem, $S^0_{f}$ and so $S_{f'}$ are $\lambda$-contracting on the intersection of $V_i^0$ with the $\epsilon$-ball centered at $i$, and send this ball into a $\lambda\cdot \epsilon$-ball, for any $\epsilon>0$.

Since $S_f(i)=i$, by continuity of $f'\mapsto S_{f'}$ and by restricting $V_f$, for every $f'\in V_f$, 
$S_{f}$ sends $cl(B_{C^0}(i,\epsilon))$ into $B_{C^0}(i,\epsilon)$.

\subsection{The $1$-jet space}\label{r=1}

Contrarily to what happened for the $C^0$-topology, the map $S_f$ is not necessarily contracting in the $C^1$-topology, even when $L$ is compact (and so $S=S^0$). However, we are going to show that the backward action of $Tf$ on the Grassmannian of $d$-planes of $T\mathcal F$ is contracting, at the neighborhood of the distribution induced by $T\mathcal L$. Actually, it is the $C^r$-endomorphism of a neighborhood of the zero section $U^*$ of $(F,\mathcal F)$,  \[\hat f' \; :\; z\in U^*\mapsto I^{-1}_{f^*\circ \pi(z)}\circ f'\circ I(z)\]
 defined as in section \ref{r=1c} which acts on the Grassmannian of $T\mathcal F$ by backward image.
 
Let us define a norm and a vector structure on a neighborhood of $T\mathcal L$ in the Grassmannian of $d$-planes of $T\mathcal F$. Let $\chi_0$ be a continuous, horizontal distribution of $d$-planes of $(F,\mathcal F)\rightarrow (L,\mathcal L)$. By horizontal we mean that for every $y\in F$, the $d$-plane $\chi_0(y)$ is sent onto $T_{|\pi(y)}\mathcal L$ by $T_{y}\pi$.

 Let $\Psi\; :\; (y,t)\in F\times \mathbb R\mapsto t\cdot y\in F$. The map $\Psi$ is a $C^r$-morphism of $(F,\mathcal F)\times \mathbb R$ onto $(F,\mathcal F)$. Let $\chi$ be the horizontal distribution defined by:
\[\chi(y)=\left\{\begin{array}{cl} T\Psi_{(y,\|y\|)}\Big(\chi_0\big(\frac{y}{\|y\|}\big)\Big)& \mathrm{if}\; y\not=0\\
T_y \mathcal L& \mathrm{if} \; y=0\end{array}\right.\]
For all $x\in L$ and $y\in F_x$, the vector space $T_x\mathcal L^*\otimes F_x$ of linear maps from $T_x\mathcal L$ to $F_x$ is isomorphic to the space $\chi(y)^*\otimes F_x$, by the isomorphism which associates to $l\in T_x\mathcal L^*\otimes F_x$ the map $l\circ T\pi_{|\chi(y)}$.

As $F_x$ is canonically isomorphic to the subspace of $T_y\mathcal F$ tangent to $F_x$, via the graph map we shall identify $\chi(y)^*\otimes F_x$ (and so $T_x\mathcal L^*\otimes F_x$) to a neighborhood of $\chi(y)$ in the Grassmannian of $d$-planes of $T_y\mathcal F$. We denote by $P^1_y$ such a vector space that we endow with the norm subordinate to those of $T_x\mathcal L$ and $F_x$. 
We denote by $P^1$ the vector bundle over $F$ whose fiber at $y\in F$ is $P^1_y$.\\

By normal expansion, for all $f'$ $C^1$-close to $f$ and for all small $x\in \pi^{-1}(V)$ for any small $d$-plane $l\in P^1_{\hat f'(x)}$, the preimage by $T_x\hat f'$ of the $d$-plane $l$ is a small $d$-plane $l'\in P^1_x$.
    
Let us show the following lemma:
\begin{lemm}
For all small $\epsilon>0$, and then $V_f$ small enough, for all $f'\in V_f$, $x\in \pi^{-1}(V)$ and $l\in P^1_{\hat f'(x)}$ both with norm at most $\epsilon>0$, the norm of $\phi_{f'x}(l):=l'\in P^1_x$ is then less than $\epsilon$. Moreover, the map $\phi_{f'x}$ is $\lambda$-contracting.\end{lemm}
\begin{proof}
Let $\pi_v$ be the projection of $T\mathcal F$ onto $F$ parallelly to $\chi$. Let $\pi_h$ be the projection of $T\mathcal F$ onto $\chi$ parallelly to $F$.

For any vector $e'\in \chi(x)$, the point $(e',l'(e'))$ is sent by $T_x\hat f'$ to 
\[\big(Tf'_h(e',l'(e')),Tf'_v(e',l'(e'))\big), \quad \mathrm{with}\; Tf'_h:= \pi_h \circ T\hat f'\; \mathrm{and}\; Tf'_{v}:= \pi_v\circ T\hat f'.\]

Let $e:=  Tf'_h(e',l'(e'))$. By definition of $l'$, the point $(e',l'(e'))$ is sent by $T_xf'$ to
$(e,l(e))$.

Therefore, we have $l(e)= Tf'_v(e',l'(e'))$ and $l(e)=l\circ Tf'_h(e',l'(e'))$.
\[\Rightarrow Tf'_v(e',l'(e'))= l\circ Tf'_h(e',l'(e'))\]
\[\Rightarrow (Tf'_v-l\circ Tf'_h)(l'(e'))=(l\circ Tf'_h-Tf'_v)(e')\]
By normal expansion, for $\epsilon$ and $V_f$ small enough, the map $(Tf'_v-l\circ Tf'_h)_{|F}$ is bijective. Consequently:
\[l'(e')=(Tf'_v-l\circ Tf'_h)_{|F}^{-1}(l\circ Tf'_h-Tf'_v)(e').\]
Hence, the expression of $l'=\phi_{f'x}(l)$ depends algebraically  on $l$, and the coefficients of this algebraic expression depends continuously on $f'$ or $x$, with respect to some trivializations.

When $f'$ is equal to $f$ and $x$ belongs to the zero section of $F_{|V}$, the map 
\[\phi_{fx} \;:\;l'\mapsto (Tf_v-l\circ Tf_h)_{|F}^{-1}(l\circ Tf_h)\]
is $\lambda$-contracting for $l$ small, by normal expansion.

Thus, for $\epsilon$ small enough, for all $x\in B_{F|V}(0,\epsilon)$ and $l \in B_{P^1_{\hat f(x)}}(0,\epsilon)$, the tangent map of $\phi_{fx}$ has a norm less than $\lambda$.
Therefore, by taking  $\epsilon$ slightly smaller, for $V_f$ small enough, the tangent map $T\phi_{f'x}$ has a norm less than $\lambda$ and hence $\phi_{f'x}$ is a $\lambda$-contraction on $cl\big({B}_{P_{\hat f(x)}}(0,\epsilon)\big)$. 

As, for $x\in V$, the map $\phi_{fx}$ vanishes at $0$, for $V_f$ small enough, the norm $\phi_{f'x}(0)$ is less than 
\[(1-\lambda)\cdot \epsilon.\]

Consequently, by $\lambda$-contraction, the closed $\epsilon$-ball centered at the $0$-section is sent by $\phi_{f'x}$ into the $\epsilon$-ball centered at $0$, for all $x\in B_{F|V}(0,\epsilon)$ and $f'\in V_f$.
 \end{proof}

For any $C^1$-section $i'$ of $(F,\mathcal F)\rightarrow (L,\mathcal L)$, the tangent space to the image of $i'$ at $i'(x)$, for every $x\in L$, is an element $\nabla i'(x)$ of $P^1_{i'(x)}$ and so of $T_x\mathcal L^*\oplus F_x$.

By conclusion 2 of Lemma \ref{lem6'}, for $x_0\in V$, for $f'\in V_f$, for $j\in V_i^0\cap \Gamma$, the $d$-planes $\nabla j(x_0)$ is sent by $T_x\hat {f'}$ to  $\nabla S^0(f',j)(y_0)$, with $x:= j(x_0)$ and $y_0:=\pi\circ \hat f'(x)$. Thus $\nabla S^0(f',j)(y_0)$ is equal to $\phi_{f'x}(\nabla j(y_0))$.

We remind that the interest of such a distribution $\chi$ is that for any morphism $p \in Mor^1 (\mathcal L,\mathbb R^+)$, we have: 
\begin{equation}\label{nabla2}\nabla(p \cdot i')(x)=T p(x)\cdot i'(x)+p(x)\cdot \nabla i'(x).\end{equation}    
Such equality is proved in section \ref{r=1c}, Equation (\ref{nabla}).

Thus we have for any $x_0\in V$:
\[\nabla S_{f'}(j)(x_0)=\rho(x_0)\cdot \phi_{f'x} (\nabla j(y_0)) +T\rho (x_0)\cdot z,\quad \mathrm{with}\; z:=S_{f'}(j)(x_0)\; \mathrm{and}\; y_0:=\pi\circ \hat {f}'(x)\]

In other notations: 
\[\nabla S_{f}(j)(x_0)=\phi'_{f'z}(\nabla j(y_0)),\quad \mathrm{ 
with}\;  \phi'_{f'z}\; :\; l\in P^1_{\hat f'(z)}\mapsto \rho \circ \pi(z)\cdot \phi_{f'z} (l) +T_{\pi(z)}\rho\cdot z\in P^1_{z},\]
 for any $z\in B_{F|V}(0,\epsilon)$, for $z$ small enough.  

We notice that $\phi'_{f' z}$ is as much contracting as $\phi_{f'z}$ and sends the $\epsilon$-ball of $P^1_{\hat f(z)}$ centered at 0 into the $\epsilon$-ball of $P^1_{z}$ centered at $0$.

\subsection{Proof of Theorem \ref{Main} when $r=1$ and in the normally expanded case}
Let us  begin by proving the existence of closed neighborhood $V_{i}$ of $i\in \Gamma$ sent by $S_{f'}$ into itself, for any $f'\in End^r(M)$ close to $f$.

For $\epsilon>0$, then $\epsilon'>0$ and finally $V_f$ small enough, the following set is included in $V_i^0$: 

\[V_i:= 
\Big\{i'\in \Gamma; \quad d_{0}(i,i')\le \epsilon',\;\mathrm{and}\; \; \|\nabla i'(x)\|\le \epsilon,\; \forall x\in V\Big\}.\]

 Therefore, by the  
two last steps, $S_{f'}$ sends $V_i$ into itself, for every $f'\in V_f$. 

We already showed that $S$ is $C^0$-contracting for the $C^0$-topology.
\begin{lemm}\label{lemm 7.3} For every $f'\in V_f$, for every $\delta>0$, there exists $N>0$ such that: for all $i',i''\in S^{N}(V_i)$, for every $x\in L$:
\[d(\nabla i'(x), \nabla i''(x))<\delta.\]\end{lemm}

Therefore, the sequence $(S_{f'}^{k}(i))_k$ converges in $Im^1(\mathcal L,M)$ to some $C^1$-immersion $i'\in V_i$. Such an immersion is unique by the $C^0$-contraction of $S_{f'}$ and so, is a fixed point of $S_{f'}$. We define the $C^1$-morphism $f'^*:= x\mapsto \pi\circ \hat {f}'\circ i'(x)$ from $\mathcal L_{|L'}$ to $\mathcal L$. Such a morphism is equivalent to $f^*_{|L'}$, and satisfies with $i(f')$ conclusions $(1)$ and $(2)$ of Theorem \ref{Main}. 

\begin{proof}[Proof of Lemma \ref{lemm 7.3}]
 For any $\delta>0$, there exists $N'\ge 0$ such that $\lambda^{N'}\cdot \epsilon$ is less than $\delta$.
  
For $k\ge 0$, let $U_k$ be the closure of $\cup_{i'\in V_i} S^k_{f'}(i')$. Thus $U_{N'}$ is identified to a compact subset of $F$. 

Let $\sigma^{\epsilon,k}$ be the space of continuous sections of the bundle $P^1_{|U_k}$ over $U_k$ whose fiber at $x\in U_k$ consists of planes of $P^1_x$ with norm less than $\epsilon$.

The following map is well defined and contracting:
 \[f'^{\#}:\sigma^{\epsilon,0}\longrightarrow \sigma^{\epsilon, 1}\]
\[ \sigma \longmapsto \left[x\in U_1\mapsto \left\{\begin{array}{cl}
 \phi_{f'x}\Big(\sigma\circ \hat f\big(x/\rho \circ \pi(x)\big)\Big)& \mathrm{if} \; \rho \circ \pi(x)\not=0\\
 0& \mathrm{else}\end{array}\right.\right]\]
 
 Thus there exists a continuous section $\sigma_{f'}\in \sigma^{\epsilon,N'}$ in $f'^{\#^{N'}}(\sigma^{\epsilon,0})$.
 Moreover, the diameter of $\sigma^{\epsilon,N'}$ is less than $\delta$ for the uniform norm induced by the one of $P^1$.
 
 By continuity of $\sigma_{f'}$ and compactness of $U_{N'}$, there exist $e>0$ such that for every $x\in V$, $(y,y')\in U_{N'}\cap F_x$, if $\|y-y'\|<e$ then $\|\sigma_{f'}(y)-\sigma_{f'}(y')\|$ is less than $\delta$.
 
 Let $N\ge N'$ such that $\lambda^N\cdot \epsilon$ is less than $e$. Thus the $C^0$-diameter of $S^N_{f'}(V_i)$ is less than $e$.
 
 Consequently, for $(i',i'')\in S^{N}_{f'}(V_i)$, the norm $\|\nabla i'(x)-\nabla i''(x)\|$ is less than:
\[d\Big(\nabla i'',\sigma_{f'}\big(i''(x)\big)\Big)+ 
d\Big(\sigma_{f'}(i''(x)), \sigma_{f'}(i'(x))\Big)+
d\Big(\sigma_{f'}\big(i'(x)),\nabla i'(x)\Big)\]
\[\Rightarrow d(Ti''(T_x\mathcal L),Ti'(T_x\mathcal L))\le 3\delta\]
 The last inequality concludes the proof of the lemma.
\end{proof}
Let us prove conclusion $(3)$ of Theorem \ref{Main}.

Let $(y_n)_{n\ge 0}\in M^{\mathbb N}$ be a $f'$-orbit which is $\epsilon$-close to the image by $i$ of a $\epsilon$-pseudo-orbit $(x_n)_{n\ge 0}\in L'^{\mathbb N}$ of $f^*$ respecting the plaques. 
Let $y_n':= (I_{x_n}^\eta)^{-1}(y_n)$ and $z_n:= \pi(y'_n)$.

We want to show that $i(f')(z_n)$ is equal to $y_n$, for every $n\ge 0$.

Let $n> 0$ be such that $i(f')(z_n)$ and $y_n$ are at distance in $F_{z_n}$ at least:
\begin{equation}\label{eq pour conclusion 3} \frac{\lambda+1}{\lambda}\sup_{j> 0}d_{F_{z_j}}\big(i(f')(z_j),y_j\big).\end{equation}

For $\epsilon>0$ small, the sequence $(z_j)_{j\ge 0}$ is close to $L'$. Also we can construct a section $i'\in V_i$ equal to $y_j$ at $z_j$ for $j> 0$, and such that the $C^0$-distance between $i'$ and $i(f')$ is the one between $\sup_{j> 0}d_{F_{z_j}}\big(i(f')(z_j),y_j\big)$.

We remind that on $V'\subset V$ the maps $S$ and $S^0$ are equal, since $\rho$ is there equal to 1.
 
By contraction of $S^0$, the $C^0$-distance between $i(f')_{|V'}$ and $S^0(f',i')_{|V'}$ is less than $\lambda\cdot d(i(f')(z_n),y_n)$.

By conclusion $(2)$ of Lemma \ref{lem6}, $S^0_{f'}(i')$ sends $(z_j)_{j\ge 0}$ to $(y_j)_{j> 0}$. By Equation (\ref{eq pour conclusion 3}), for every $j> 0$, the point $z_j$ is sent by $i(f')$ to $z_j$. 
By normal expansion, the point $z_0$ is sent by $i(f')$ to $z_0$.

\subsection{General case: $r>1$}
Unfortunately, as we deal with maps which have possibly singularities along the leaves, the action of $Tf$ on the Grassmannian is not as well defined as in Theorem \ref{Maincontract}. I do not know any other way to prove this result than the following calculus.

As we deal with the $C^r$-topology, we shall generalize the Grassmannian concept as follow:

For $x\in U^*$, let $G_x^r$ be the set of the $C^r$-$d$-submanifolds of $M$ which contain $I(x)$, quotiented by the following $r$-tangent relation:

Two such submanifolds $N$ and $N'$ are equivalent if there exists a chart $(U,\phi)$ of a neighborhood of $I(x)\in M$, which sends $N\cap U$ onto $\mathbb R^{d}\times \{0\}$ and sends $N'\cap U$ onto the graph of a map from $\mathbb R^{d}$ into $\mathbb R^{n-d}$, whose $r$-first derivatives vanish at $\phi(x)$.

We notice that for $r=1$, this space is the Grassmannian of $d$-planes of $T_x\mathcal F\approx T_{I(x)}M$.

As we are interested in the submanifolds ``close'' to the embedding of small $\mathcal L$-plaques by $i$,  we restrict our study to the $d$-submanifolds containing $I(x)$ and transverse to $F_{\pi(x)}^\epsilon$.

The preimage of a submanifold of this equivalent class by the map $\exp\circ T_xI$ is a graph of a $C^r$-map $\overline{l}$ from $\chi(x)$ to $F_{\pi(x)}$, with $\exp$ the exponential associated to a Riemannian metric of $M$.

Moreover, its $r$-tangent equivalence class can be identified to the Taylor polynomial of $\overline{l}$ at $0$:
\[\overline{l}(u)=T_0\overline{l}(u)+\frac{1}{2}T^2_0\overline{l}(u^2)+\cdots +\frac{1}{r!} T^r_0\overline{l}(u^r)+o(\|u\|^r),\]
where $u$ belongs to $\chi(x)$, the $k^{th}$-derivative $T_0^k\overline{l}$ belongs to the space $L_{sym}^k(\chi(x),F_{\pi(x)})$ of $k$-linear symmetric maps from $\chi(x)^k$ into $F_{\pi(x)}$. We notice that we abused of notation by writing $u^k$ instead of $(u,\dots,u)$.

The map $l(u):= \sum_{k=1}^n \frac{1}{k} T^k_0\overline{l}(u)$ is an element of the vector space:
\[P_x^r:= \oplus_{k=1}^r L_{sym}^k (\chi(x),F_{\pi(x)}).\]

Conversely, any vector $l\in P_x^r$ that we write in the form:
\[l\::\;u\in \chi(x)\mapsto \sum_{k=1}^r l_k(u^k)\]
is the class of the following $C^r$-$d$-submanifold of $M$:
 \[\exp\circ TI \Big(\{(u+l(u));\; u\in \chi(x)\; \mathrm{and} \; \|u\|<r_i(x)\big\}\Big),\]
 where $r_i$ depends continuously on the injectivity radius of $\exp$.
 
The linear map $l_1$ from $\chi(x)$ to $F_{\pi(x)}$ is called the \emph{linear part of $l$}. We notice that $l_1$ belongs to $P^1_x$.

We denote by $P^r$ the vector bundle over $U^*$, whose fiber at $x$ is $P^r_x$.

 By normal expansion, for $U^*$ small enough and $f'$ close enough to $f$, for any point $x\in U^*$ sent by $\hat f'$ to some $y\in U^*$, any $l\in P^r_y$ whose linear part is small enough, the preimage  by $f'$ of a representative of $l$ is a representative of a vector $\phi_{f'x}(l)\in P^r_x$, which depends only on $l$.

Let us show the following lemma:
\begin{lemm}\label{Cr:contractivity}
For every $\epsilon>0$ small enough and then  $V_f$ small enough, for all $f'\in V_f$, $x\in \pi^{-1}(V)\cap U^*$ and $l\in P^r_{\hat f'(x)}$ with linear part of norm at most $\epsilon>0$, the norm of the linear part of $\phi_{f'}(l)$ is then less than $\epsilon$. Moreover, the map $\phi_{f'}$ is of the form 
\[(l_m)_m\in P^r_{\hat f(x)}\mapsto \Big[C_m^{f'}(l_m)+\Psi_m\big((l_k)_{1\le k< m}\big)\Big]_m\in P^r_x\]
where $\Psi_m\big((l_k)_{1\le k\le m}\big)$ does not depend on $(l_i)_{i> m}$. Also $C_m^{f'}$ is a $\lambda$-contracting endomorphism of $L_{sym}^m(\chi(x),F_{\pi(x)})$ endowed with the norm subordinate to $\chi(x)$ and $F_{\pi(x)}$, and $\chi(x)$ is endowed with the norm induced by $T_x\mathcal F$.\end{lemm}

For $i'\in \Gamma$ and $l\in \{1,\dots ,r\}$, let $\nabla^l i'$ be the section $\pi_v\circ T^li'$, where $\pi_v$ is the projection of $T\mathcal F$ onto $F$ parallelly to $\chi$ and $T^li'$ is the $l$-tangent maps of $i'$. Thus $\nabla^l i'(x)$ is a vector of $L_{sym}^l(T_x\mathcal L,F_x)$, for every $x\in L$.

As in section \ref{r=1c}, we have for any $i'\in \Gamma$:
\[\nabla^l (\rho\cdot i')(x)=\sum_{k=0}^l C_l^k T^{l-k}\rho\cdot \nabla^k i'(x).\]
Thus, $\nabla^l(S_{f'}(i'))(x)=\rho(x)\cdot \nabla^l S^0(f',i')(x)+ \Psi_m\big((\nabla^k S^0(f',i')(x))_{1\le k<l}\big)$, where $\Psi'_m$ is a continuous function which does not depend on $\big(\nabla^k S^0(f',i')(x)\big)_{k\ge l}$. 

On the other hand, by conclusion $2$  of Lemma \ref{lem6} and by Lemma \ref{Cr:contractivity}, for any $x\in V$ and $i'\in V_i^0$, $(\nabla^k S^0(f',i')(x_0))_{k=1}^r$ is equal to the image by $\phi_{f'x}$ of $(\nabla^k i'(y_0))_{k=1}^r$, with $x_0\in V$, $x= S^0(f',i')(x_0)$ and $y_0:=\pi\circ \hat f'\circ S^0(f',i')(x_0)$.

Thus $(\nabla^k S_{f'}(i')(x_0))_{k=1}^r$ is equal to the image of $\big((\nabla^k i')(y_0)\big)_{k=1}^r$ by a continuous function $\phi'_{f'x}$ of the form:
\[\phi'_{f'x}:\;(l_m)_m\in P^r_{\hat f(x)}\mapsto \Big[\rho(x)\cdot C_{m,x}^{f'}(l_m)+\Psi'_m((l_k)_{k=1}^m)\Big]_m\] 
Such a function is contracting for a good norm on $P^r$ independent on the base point. To conclude the proof, we proceed as in the previous section.
  
\begin{proof}[Proof of Lemma \ref{Cr:contractivity}]
We showed in the case $r=1$ that $\phi_{f'x}$ preserves the subset of vectors of $P^1$ with norm at most $\epsilon$, and hence preserves the subset of $P^r$ formed by maps of linear part at most $\epsilon$. Let us show the $\lambda$-contraction of $\phi_{f'x}$. 

Let $l':=\phi_{f'}(l)$. Let $J_z^rf'$ be the $r$-jet map of $f'$ at $z:= I(x)$ (see \cite{Michor}).

We remind that the $r$-jet $J_z^rf'$ of $f'$ at $z$ is a vector of 
\[\bigoplus_{j=1}^r L_{sym}^j (T_zM, T_{f'(z)}M)\] such that, if we denote by $f_j'$ its component in 
$L_{sym}^j(T_zM,T_{f'(z)}M)$, we have:

\[\exp^{-1}_{f'(z)}\circ f'\circ \exp_z(u)=\sum_{j=1}^r f'_j(u^j) + o(\| u\|^r),\; \mathrm{for} \; u\in T_z M\]

Let $J_x^r \hat f':= (T_{\hat f'(x)}I)^{-1}\circ J_z^rf'\circ T_x I$ and $J_x^r \hat f'=: \sum_{j=1}^r \hat f'_j\in  \bigoplus_{j=1}^r L_{sym}^j(T_x \mathcal F, T_{\hat f'(x)} \mathcal F)$. 

By definition of $l':= \phi_{f'x}(l)$, for any $u'\in \chi(x)$, there exists $u\in \chi(\hat f'(x))$ such that:

\begin{equation}\label{mise en valeur du jet de f}
J_x^r \hat f'(u'+l'(u'))=u+l(u)+o(\|u\|^r).\end{equation} 
  
We recall that $\pi_v$ and $\pi_h$ denote the projection of $T\mathcal F$ onto respectively $F$ and $\chi$.

By (\ref{mise en valeur du jet de f}), we have:
\[u:= \pi_h\circ J^r_x\hat f'(u'+l'(u'))+o(\|u'\|^r)\; \mathrm{and}\; l(u)=\pi_v\circ J^r_x\hat f'(u'+l'(u'))+o(\|u'\|^r).\]

Thus, we have:
\begin{equation}\label{1eq}
l\circ \pi_h\circ J^r_x\hat f'(u'+l'(u'))=\pi_v\circ J^r_x\hat f'(u'+l'(u'))+o(\|u'\|^r).\end{equation}
We have:
\begin{equation}\label{2eq} J^r_x \hat f'(u'+l'(u'))=\sum_{I\in R} \hat f'_{|I|}\Big[\prod_{k\in I} l'_k(u'^k)\Big]+o(\|u'\|^r),\end{equation}
where $R$ is the set $\cup_{k=1}^r\{0,\dots ,r\}^k$;  $l'_0(u'^0)$ is equal to $u'$;  and for $I\in R$, $|I|$ denotes the length of $I$. 

Let $f'_{kv}$ and $f'_{kh}$ be respectively the linear maps $\pi_v\circ \hat f'_k$ and $\pi_h\circ \hat f'_k$ respectively, for every $k\in \{0,\dots, r\}$.

It follows from Equations (\ref{1eq}) and (\ref{2eq}) that:
\begin{equation}\label{agd}l\Big(\sum_{I\in R} f'_{|I|h}\big[ \prod_{k\in I} l'_k(u'^k)\big]\Big)
=\sum_{I\in R} f'_{|I|v}\big[\prod_{k\in I}l_k'(u'^k)\big]+o(\|u'\|^k)\end{equation}

On the one hand, we have:

\begin{equation}\label{ad}
\sum_{I\in R} f'_{|I|v}\big[\prod_{k\in I}l_k'(u'^k)\big]=\sum_{m=1}^r\left[ \sum_{I\in R,\; \Sigma I=m}f'_{|I|v} \big[\prod_{k\in I}l'_k(u'^k)\big]\right]+o(\|u'\|^r)\end{equation}

 with, for every $I\in  R$, the integer $\Sigma I$ equal to $\sum_{j\in I} j$ plus the number of times that $0$ appears in $I$.
 
On the other, as 

\[l\Big(\sum_{I\in R} f'_{|I|h}\big[ \prod_{k\in I} l'_k(u'^k)\big]\Big)=
 \sum_{a=1}^r l_a\Big(\sum_{I\in R} f'_{|I|h}\big[ \prod_{k\in I} l'_k(u'^k)\big]\Big)^a\]
and as, 
\[\Big(\sum_{I\in R} f'_{|I|h}\big[ \prod_{k\in I} l'_k(u'^k)\big]\Big)^a=\sum_{(I_\alpha)_\alpha\in R^a} 
\prod_{\alpha=1}^a f'_{|I_\alpha|h}\big[\prod_{k\in I_\alpha} l_k'(u'^k)\big]\]
the polynomial map $l\Big(\sum_{I\in R} f'_{|I|h}\big[ \prod_{k\in I} l'_k(u'^k)\big]\Big)$ is equal to: 
\begin{equation}\label{ag}
\sum_{m=1}^r \quad 
\sum_{\{(I_\alpha)_{\alpha\in A}\in R^*,\; \sum_\alpha\Sigma I_\alpha=m\}} 
l_{|A|} \Big[\prod_{\alpha\in A} f'_{|I_\alpha|h}\big[\prod_{k\in I_\alpha} l_k'(u'^k)\big]\Big]\end{equation}
with $R^*:= \cup_{a=1}^r R^a$.

By identification, it follows from equations  (\ref{agd}), (\ref{ad}) and (\ref{ag}) and that, for every $m\in \{1,\dots, r\}$:

\[ \underbrace{\sum_{\{(I_\alpha)_{\alpha\in A}\in R^*,\; \sum_\alpha\Sigma I_\alpha=m\}} 
l_{|A|} \Big[\prod_{\alpha\in A} f'_{|I_\alpha|h}\big[\prod_{k\in I_\alpha} l_k'(u'^k)\big]\Big]}_{l'_m\; \mathrm{only\; occurs\; for}\; (I_\alpha)_\alpha=((m));\;  l_m\;  \mathrm{only\; for}\; (I_\alpha)_\alpha\in \{\{0\},\{1\}\}^m}=
\underbrace{ \sum_{I\in R,\; \Sigma I=m}f'_{|I|v} \big[\prod_{k\in I}l'_k(u'^k)\big]}_{\mathrm{Here} \;l'_m\; \mathrm{only\; occurs\; for}\; I=(m).}\]
Thus,  there exists an algebraic function $\phi$, such that $f'_{1v}\circ l'_m(u'^m)$ is equal to:
\[
l_1\circ f_{1h}'\circ l'_m(u'^m)+ \sum_{(i_\alpha)_{\alpha=1}^m\in \{0,1\}^m} l_m\Big(\prod_{\alpha=1}^m f_{1h}'\circ l_{i\alpha}'(u'^{i_\alpha})\Big)+\phi\big((l_i)_{i<m},(l_i')_{i<m},(f'_i)_{i=1}^r\big)\]

Since the linear part $l_1$ of $l$ is small, we have:
 \[l_m'=(f'_{1v}-l_1\circ f'_{1h})_{|F_{\pi(x)}}^{-1}
 \Big[  
 \sum_{I\in \{0,1\}^m} l_m\circ \prod_{k\in I} f'_{1h}\circ l_k+\phi\big((l_i)_{i<m},(l_i')_{i<m},(f_i)_{i=1}^r\big)\Big].\]
For $x\in V$ and $f'=f$, we have $f'_{1h|F_x}=0$.

Thus,  
\[\sum_{I\in \{0,1\}^m} l_m\prod_{k\in I} f'_{1h}\circ l_k=l_m\circ  \big(f'_{1h} \big)^m.\] 
It follows from the $r$-normal expansion that the map
\[C\;:\; l_m\mapsto  (f'_{1v}-l_1\circ f'_{1h})_{|F_x}^{-1} \circ l_m\circ  \big(f'_{1h}\big)^m\]
is $\lambda$-contracting, when $l_1$ is small.

Also, the map $l_s'$ is an algebraic function of only  $(l_k)_{k\le s}$ and $(f'_k)_{j\le r}$, for $s\le m$.
 Thus, 
 \[l_m'=C_m(l_m)+\phi\big((l_i)_{i<m},(f_i')_i\big)\]

This implies that for a norm on $P^r$, the map $\phi_{fx}$ is contracting, for $U^*$ small enough and then $f'$ $C^r$-close to $f$.
 
\end{proof}
\section{Proof of persistences of complex laminations by deformation}
Let us prove the persistence of $0$-normally expanded complex laminations by deformation. The persistence of $0$-normally contracted laminations is proved similarly.

Let $(F,\mathcal F,\pi,I)$ be a differentiable tubular neighborhood of the none necessarily compact lamination $(L,\mathcal L)$ immersed by $i$.

Let $L'$ be a precompact open subset of $L$ whose closure is sent into $L'$ by the pullback $f^*$. Let $B_0$ be a precompact, open subset of $B$, that we will restrict. 

Let $\hat F$ be the product $B_0\times F_{|L'}$ and $\hat \pi\; : \;  (x,t)\in B_0\times F_{|L'}\rightarrow (x,t)\in B_0\times L'$. 

Let $Z$ be the set of maps of the form:
\[i_\sigma\; :\; (t,x)\in B_0\times L'\mapsto (t,I\circ \sigma(t,x))\in B_0\times M,\]
where $\sigma$ is a $C^0$-bounded, differentiable section of $\hat \pi \; :\; \hat F\rightarrow B_0\times L'$, such that $i_\sigma$ sends any small plaque of the product lamination $B_0\times \mathcal L_{|L'}$ onto a complex submanifold of $B_0\times M$.

 We notice that $Z$ is not empty since it contains $i_0\; :\; (t,x)\mapsto (t,i(x))$. 

\begin{lemm} The subspace $Z$ equipped with the $C^0$-norm:
\[\|i_\sigma\|= \sup_{(t,x)\in B_0\times L'} \|\sigma(t,x)\|\]
is a complete space\end{lemm}
\begin{proof} This follows from the Cauchy integral theorem.\end{proof}
Let $\hat S\; :\; i_\sigma\in Z\longmapsto \Big[(t,x)\mapsto \big(t,S^0(f_t,i_\sigma(t,\sigma(t,\cdot))\big)\Big]$, with $S^0$ defined by Lemma 
\ref{lem6} with $r=1$. By restricting $B_0$, such a map is well defined since this lemma only uses the $0$-normal expansion.

By conclusion \ref{preconc4} of Lemma \ref{lem6}, the image $V'$ by $\hat S(i_\sigma)$ of a small plaque of $B_0\times\mathcal L_{|L'}$ is sent by $\hat f\; :\; (t,x)\in B_0\times M\mapsto (t,f_t(x))$ into the image $V$ by $i_\sigma$ of a small plaque of $B_0\times\mathcal L_{|L'}$.  By $0$-normal expansion, the map $\hat f$ is transverse to $V$ at $V'$:
\[\forall y\in V',\; T\hat f\big(T_y (B_0\times M)\big)+T_{\hat f(y)}V=T_{\hat f(y)}(B_0\times M).\]
Thus, $V'$ is an open subset of $\hat f^{-1}(V)$, since $V$ and $V'$ have the same dimension. As $\hat f$ and $V$ are complex analytic, $V'$ is also complex analytic. Thus, $\hat S(i_\sigma)$ belongs to $Z$.
Finally, by restricting $B_0$, the map $\hat S$ is $C^0$-contracting, since subsection \ref{C0contraction} needs only the $0$-normal expansion.

Consequently, $\hat S$ has a unique fixed point $\underline{i}\; : \; (t,x)\in B_0\times L'\mapsto (t,i_t(x))$ in $Z$.
By Lemma \ref{J2},  $B_0\times M$  induces via $\underline{i}$ a unique complex laminar structure $\mathcal D$ on $D$, compatible with the $C^1$-laminar structure and  the complex structure $B_0\times \mathcal L_{|L'}$. By compatibility of the structure, the projection  $B_0\times L'\mapsto B_0$ induces a holomorphic submersion $\overline{w}\; :\; (D,\mathcal D)\rightarrow B_0$. We notice that $S(f,i)=i$ and that $\underline{i}_{|\{t_0\}\times L'}$ is $i$: by uniqueness stated in Lemma \ref{J2}, the structure $\overline{w}^{-1}(t_0)=(L',\mathcal L_{t_0})$ is equal to $(L',\mathcal L_{|L'})$.

We notice that $f_t$ preserves the immersed lamination $(L,\mathcal L_t)=\overline{w}^{-1}(t)$ by $i_t= \underline{i}_{|\overline{w}^{-1}(t)}$ since $\underline{i}$ is a fixed point of $S$ and hence since $S^0 (f_t,i_t)$ equals $i_t$.

\section{Proof of persistences of normally hyperbolic laminations}

\subsection{Existence of laminar structures on the stable and unstable sets of normally hyperbolic laminations}  
\begin{prop}\label{trans}
Under the hypotheses of Theorem \ref{Main} in the normally hyperbolic case, there exist two laminations $(L^s,\mathcal L^s)$ and $(L^u,\mathcal L^u)$ $C^r$-immersed by respectively $i^s$ and $i^u$ into $M$, such that:
\begin{itemize}
\item[-] $f$ $r$-normally expands $(L^s,\mathcal L^s)$ over a morphism $f_s^*$ from a the restriction of $\mathcal L^s$ to an open,  precompact subset $L'^s$ into $\mathcal L^s$. 
\item[-] $f$ $r$-normally contracts $(L^u,\mathcal L^u)$ over an injective, open immersion $f_u^*$ from the restriction of $\mathcal L^u$ to a precompact open subset $D^u$ to $\mathcal L^u$,
\item[-] there exist two canonical $(C^r)$-inclusions of $(L,\mathcal L)$ into $(L^s, \mathcal L^s)$ and $(L^u,\mathcal L^u)$ such that the following diagram is well defined and commutes:
\[\begin{array}{rclll}
     L'^s     &\stackrel{f_s^*}{\rightarrow} &L^s         &                &   \\
     \uparrow&                              &\uparrow    &\stackrel{i^s}{\searrow}                   &   \\
      L'      &\stackrel{f^*}{\rightarrow}   &L           &\stackrel{i}{\rightarrow} &M  \\
   \downarrow&                              &\downarrow  &\stackrel{i^u}{\nearrow}       &   \\     
   D^u      &\stackrel{f_u^*}{\rightarrow} &L^u         &                &   \\
\end{array}\]
\item[-] for every $\epsilon>0$ small enough, all points $x\in L'$ satisfy\footnote{We remind that $\mathcal L^{\epsilon}_x$ is the union of $\mathcal L$-plaques with diameter less than $\epsilon$ which contain $x$.}   
\[i(\mathcal L_x^\epsilon)=i^s(\mathcal L_x^{s\epsilon})\;\ovfork\; i^u(\mathcal L_x^{u\epsilon}).\]
\item[-] let $x\in L'$ with $f^*$-forward orbit included in $L'$. Then a local strong stable manifold of $i(x)$ is included in the immersion by $i^s$ of a $\mathcal L^s$-plaque of $x$,
\item[-] let $x\in L'$ with $f^*$-backward orbit included in $L'$. Then a local strong unstable manifold of $i(x)$ is included in the immersion by $i^u$ of a $\mathcal L^u$-plaque  of $x$,
\end{itemize} 
   
\end{prop}
\begin{proof}[Proof of Proposition \ref{trans}]

We remind that $r$ is positive. 

\noindent \underline{ Definition of $(L^u,\mathcal L^u)$}\\ 
Since $f^*$ is bijective, we can suppose $Tf$-invariance of the strong unstable direction  $E^u$.

we notice that the section of the Grassmannian $x\in L\mapsto E^u_x$ is continuous but not necessarily differentiable. However by Appendix A.2.1 of \cite{PB1}, there exists a $C^r$-lifting $E'^u$ of $i$ into the Grassmannian of $TM$ of class $C^r$ which is $C^{0}$-close to the section $E^u$.  By Lemma \ref{L-fibre} (see below) the section $E'^u$ defines a  structure of $C^r$-lamination $\mathcal L^u$ on the vector bundle $\pi^u\; :\; L^u\rightarrow L$ whose fiber at $x\in L$ is $E^u_x$ and such that:  
\begin{itemize}
\item[-] the leaves of $\mathcal L^u$ are the preimages by $\pi$ of the leaves of $(L,\mathcal L)$,
\item[-] $(L,\mathcal L)$ is canonically $C^r$-embedded into $(L^u,\mathcal L^{u})$ as the $0$-section of this bundle,
\item[-] for a small, smooth, positive function $\delta$ on $(L^u,\mathcal L^u)$, the map:
\[j_0\; :\; (x,u)\in L^u\mapsto \exp_x\left(\frac{\delta(x)\cdot u}{\sqrt{1+u^2}}\right)\]
is a $C^r$-immersion of $\mathcal L^u$ into $M$, with $\exp$ the exponential map associated to a Riemmanian metric of $M$..\end{itemize}
Let $(F,\mathcal F,I,\pi)$ be a tubular neighborhood of the immersed lamination $(L^u,\mathcal L^u)$ by $j_0$.
Let $V$ be a small neighborhood of $f^*(cl(L'))$ in $L^u$.
Let $\eta>0$ be small enough such that for every $x\in V$, the family of submanifolds $(F_y^\eta)_{y\in \mathcal L^{u \eta}_x}$ are the leaves of a $C^r$-foliation, with $F_y^\eta$ the image by $I$ of the ball of radius $\eta$ and centered at $0$ in $F_y$. 

 Let $\Gamma$ be the subset of the $C^r$-immersions $j$ of $(L^u,\mathcal L^u)$ such that:
\begin{itemize}
\item[-] the restriction of $j$ to $L$ is equal to $i$,
\item[-] the restriction of $j$ to the complement of $V$ is equal to $j_0$,
\item[-] the image by $j$ of $y\in L^u$ belongs to $F_y^\eta$,
\item[-] the immersion $j$ is $C^r$-close to $j_0$ in the $C^r$-topology: $\Gamma$ is a small neighborhood of $j_0$.
\end{itemize}

By normal hyperbolicity, for $E^{u'}$ close enough to $E^u$ and also $\Gamma$, and $\delta$ small enough, for every $x\in V$, for every $j\in \Gamma$, $f$ sends $j\circ \pi^{u-1}(\mathcal L_{f^{*^{-1}}\circ \pi^u(x)}^\eta)$ to a submanifold which intersects transversally at a unique point $S^0(j)(x)$ the submanifold $F_x^\eta$.
We notice that $S^0(j)$ is plaquewise a $C^r$-immersion since it is the composition of the immersion $j_0$ of a $\mathcal L^u-$plaque of $x$  with the holonomy along a $C^r$-foliation between two transverse $C^r$-sections. As these foliations and transverse sections depend continuously on $x\in V$, the map $S^0(j)$ is a $C^r$-immersion of $(V,\mathcal L^u_{|V})$ into $M$.

 Let us fix a $C^r$-morphism $\rho$ from $(L^u,\mathcal L^u)$ into $[0,1]$ with support in $V$ and equal to $1$ on a neighborhood $V'$ of $f^*(cl(L'))$.
Let $S(j)$ be equal to $j_0$ on $V^c$ and to $\rho\cdot S^0(j)$ on $V$, where the scalar product use the vector space structure of the tubular neighborhood $(F,\mathcal F,I,\pi)$. Such a map $S(j)$ is a $C^r$-immersion of $(L^u,\mathcal L^u)$ into $M$.

The existence of a fixed point $i^u$ of $S$ is proved similarly as in the proof of the persistence of normally contracted laminations, except that:
\begin{itemize} \item[-] we do not need to deal with the $C^0$-topology, since $i$ is equal to the restriction to $L$ of any immersion of the image of $S$,
\item[-] we need possibly to  take $E'^u$ closer to  $E^u$ and $V$ smaller, in order to use the contraction of some map $\phi_f$ defined on the neighborhood of $E^u\oplus T\mathcal L$ in the Grassmannian of $T\mathcal F$. \end{itemize}

Let $D^u$ be an open neighborhood of $L'$ in $L^u$, small enough in order that every point $x\in D^u$ have its image by $i^u$ sent by $f$ into the image by $i^u$ of a small plaque containing $f^{*}_u\circ \pi(x)$ and included in $V'$. 

Let $f^*_u(x)$ be the point of this plaque such that $i^u\circ f^*_u(x)=f\circ i^u(x)$.

Since the restriction of $Tf $ to $Ti(T\mathcal L)\oplus E^u$ is open and injective, for $E'^u$ close enough to $E^u$, and also $V$ small enough, $f^*_u$ is an open and injective immersion of $(D^u,\mathcal L^u_{|D^u})$ into $(L^u,\mathcal L^u)$. 

From the construction by graph transform of $i^u$, one easily shows that any points $x\in \cap_{n\ge 0} f^{* -n}(L')$ is sent by $i$ to a point with local strong unstable manifold included in the image by $i^u$  of a plaque of $\mathcal L^u$.

The existence of $i^s$ and $(L^s,\mathcal L^s)$ is proved  similarly by using this time the proof of Theorem \ref{Main} in the normally expanded case.   

Also since $\Gamma$ is a small neighborhood of $j_0$, the tangent space $Ti^u (\mathcal L^u_{|L})$ is close to $Ti(T\mathcal L)\oplus E^u$. Similarly, the tangent space $Ti^s (T\mathcal L^s_{|L})$ is close to $Ti(T\mathcal L)\oplus E^s$.   
 Therefore $i^u$ and $i^s$ are transverse at $L'$ and so the last conclusion of the proposition is proved.

\subsection{Differentiable persistence of normally hyperbolic  laminations} 
\begin{proof}[Proof of Theorem \ref{Main} when the lamination is normally hyperbolic]
We apply Theorem \ref{Main} in the normally expanded case to $(i^s,f^*_s,f,L^s)$ and Theorem \ref{Maincontract} to $(i^u,f^*_u,f,D^u)$ with some precompact, open neighborhoods $L'^s \subset L^s $ and $L'^u\subset D^u$ of $L'$. These theorems give a neighborhood $V_f$ of $f\in End^r(M)$ formed by  endomorphisms $f'$ for which there exist 
$i^s(f')\in Im^r(\mathcal L^s,M)$, $i^u(f')\in Im^r(\mathcal L^u,M)$, $f'^*_s\in Mor_{f^*_s}(\mathcal L^s_{|L'^s},\mathcal L^s)$ and $f'^*_u\in Im^r_{f^*_u}(\mathcal L^u_{|D^u}, \mathcal L^u)$ $C^r$-close to respectively $i^s$, $i^u$, $f^*_s$, $f^*_u$ such that the following diagrams commute:
\[\begin{array}{rcccl}
&&f'&&\\ 
&M&\rightarrow&M&\\
i^s(f')&\uparrow&&\uparrow&i^s(f')\\
&L'^s&\rightarrow & L^s&\\
&&f'^*_s&&\end{array}
\qquad
\begin{array}{rcccl}
&&f'&&\\
&M&\rightarrow&M&\\
i^u(f')&\uparrow&&\uparrow&i^u(f')\\
&L'^u&\rightarrow & L^u&\\
&&f'^*_u&&\end{array}.\]
Moreover $i^u(f')$ and $i^s(f')$ are equal to respectively $i^u$ and $i^s$ on the complement of compact subsets $V^s$ and $V^u$ independent of $f'\in V_f$.

For $\epsilon>0$ sufficiently small and then $V_f$ small enough, for every $x\in L\cap (V^s\cup V^u)$ and $f'\in V_f$, the intersection of $i^s(f')(\mathcal L_x^{s\epsilon})$ with $i^u(f')(\mathcal L_x^{u\epsilon})$ is transverse.
\[\mathrm{Let}\quad \mathcal L_x^{f'\epsilon}=i^s(f')(\mathcal L_x^{s\epsilon})\;\ovfork\; i^u(f')(\mathcal L_x^{u\epsilon}).\]
 We want that $i(f')$ sends $x\in L\cap (V^s\cup V^u)$ into this intersection. In order to obtain the smoothness of $i(f')$, we use a tubular neighborhood  $(F,\mathcal F,I,\pi)$  of the immersed lamination $(L,\mathcal L)$.

Thus, for $\epsilon$ small and then $V_f$ small enough, for every $y\in \mathcal L_x^\epsilon$,  the submanifold $ F_{y}^{\epsilon}:= I\big(B_{F_x}(0,\epsilon)\big)$ intersects transversally at a unique point $i(f')(x)$ the submanifold $\mathcal L_x^{f'\epsilon}$. 
  
In other words, $i(f')_{|\mathcal L_{x}^{\epsilon/2}}$ is the composition of $i_{|\mathcal L_x^{\epsilon/2}}$ with the holonomy along the $C^r$-foliation formed by the leaves $(F_{y}^{\epsilon})_{y\in \mathcal L_{x}^{\epsilon}}$, from $i(\mathcal L_{x}^{\epsilon})$ to the transverse section $\mathcal L_x^{f'\epsilon}$.
 
 Thus, the map $i(f')$ is of class $C^r$ along the $\mathcal L$-plaques.
 As these  manifolds depend $C^r$-continuously on $x\in L$, the map $i(f')_{|L\cap(V^s\cup V^u)}$ is a
 $\mathcal L$-morphism into $M$. For $x\in L\setminus (V^s\cap V^u)$ we put $i(f')(x)=i(x)$.
 
 We notice that $i(f')$ is $C^r$-close to $i$ when $f'$ is $C^r$-close to $f$, since $i^s(f')$ and $i^u(f')$ are $C^r$-close to respectively $i^s$ and $i^u$. As $i(f')$ and $i$ restricted to $L\setminus (V^s\cup V^u)$ are equal,
 by precompactness of $L\cap (V^s\cup V^u)$, for $V_f$ sufficiently small, the morphism $i(f')$ is an immersion, for every $f'\in V_f$.

 Let us construct $f'^*$ such that the diagram of Theorem \ref{Main} commutes. 
 For all $\delta>0$ and then $V_f$ small enough, for every $f'\in V_f$, 
 $i(f')$ sends every $x\in L'$ into the intersection of $i^s(f')(\mathcal L_x^{s\delta})$ with $i^u(f')(\mathcal L_x^{u\delta})$. Thus, for $V_f$ sufficiently small, by Theorems \ref{Main} and \ref{Maincontract}, every $f'\in V_f$  sends the point $i(f')(x)$ into the intersection $\mathcal L^{f'\epsilon}_{f^*(x)}$ of $i^s(f')(\mathcal L_{f^*(x)}^{s\epsilon})$ with $i^u(f')(\mathcal L_{f^*(x)}^{u\epsilon})$. Let $f'^*(x)$ be the point of $\mathcal L_{f^*(x)}^\epsilon$ such that $f'\circ i(f')(x)$ belongs to $ F_{f'^*(x)}^\epsilon$. We notice that the diagram of Theorem \ref{Main} commutes well.
 
 To see the regularity of $f'\mapsto f'^*$, we proceed as before: $f'^*$ is locally the composition of $i(f')$ with $f'$ with the holonomy along a $C^r$-foliation between two transverse sections, all of them depending continuously on the point. Thus, $f'^*$ is well a $C^r$-endomorphism of $\mathcal L$, equivalent to $f^*_{|L'}$. As $i(f')$ is close to $i$, the morphism   
$f'^*$ is close to $f^*_{|L'}$.

 \end{proof}
\subsection{Proof of persistence of normally hyperbolic complex laminations by deformations (Theorem \ref{defo-hyp})}
Let $(L^u,\mathcal L^u)$ be the lamination $C^1$-immersed by $i^u$ provided by Proposition \ref{trans}, with $r=1$ and $L=L'$. 
By Lemma \ref{J2}, the complex manifold $M$ induces on $\mathcal L'^u:=\mathcal L^u_{L'^u}$ a complex analytic structure such that $i^u$ is complex analytic, as soon as $Ti^u(T\mathcal L'^u)$ is $J$-invariant, with $J$ the automorphism of $TM$ provided by its complex structure.

To show the $J$-invariance of $Ti(T\mathcal L'^u)$, we endow $i^{u*}TM\rightarrow L^u$ with a cone field $\chi$  over the zero section of $i^{u*}( TM)$ such that:
\begin{itemize}
\item[-] $\chi$ contains $Ti^u(T_x\mathcal L'^u)$ canonically embedded into $i^{u*}TM_x$, for all $x\in L'^u$,
\item[-] $Ti^u(T_x\mathcal L'^u)$ is a maximal vector subspace of $i^{u*}TM_x$ included in $\chi$, for all $x\in L'^u$,
\item[-] the angle of $\chi$ is small.
\end{itemize}
Let $x\in  L'^u$. We want to prove that   $Ti^u(T_x\mathcal L'^u)$ is $J$-invariant. For this end, since $L$ is compact, we may suppose that $L'^u$ is preserved by $f^{*-1}$ and so 
we notice that
$Ti^u(T_x\mathcal L^u)$ is equal to the intersection:
\[\bigcap_{n\ge 0} Tf^n(\chi_{f^{*-n}(x)}).\]
Such an intersection is equal to 
\[\bigcap_{n\ge 0} Tf^n\big(Ti^u(T_{f^{*-n}(x)}\mathcal L'^u)\big).\]
As $Ti^u(T_y\mathcal L^u)$ is $J$ ,invariant for $y\in L$, for $n\ge 0$ sufficiently large $J\circ Ti^u(T_{f^{*-n}(x)}\mathcal L'^u)$ is contained in $\chi$. Thus 
\[\bigcap_{n\ge 0} Tf^n\big(J\circ  Ti^u(T_{f^{*-n}(x)}\mathcal L'^u)\big)\]
is included in $Ti^u(T_x\mathcal L'^u)$. As $Tf$ commutes with $J$, we have:
\[Ti^u(T_x\mathcal L'^u)\supset \cap_{n\ge 0} Tf^n\big(J\circ  Ti^u(T_{f^{*-n}(x)}\mathcal L^u)\big)=J\circ Ti^u(T_x\mathcal L'^u).\]
Thus, $Ti^u(T_x\mathcal L^u)$ is $J$-invariant.

To endow $(L_s,\mathcal L_s)$ with a complex analytic structure such that $i^s$ is holomorphic, we proceed similarly (actually it is even simpler). 
We finally remark that the inclusions:
\[i^s\; : \; (L,\mathcal L)\hookrightarrow (L_s,\mathcal L_s)\quad \mathrm{and}\quad 
i^u\; : \; (L,\mathcal L)\hookrightarrow (L^u,\mathcal L^u)\]
 are complex analytic by uniqueness of the complex structure in Lemma \ref{J2}.

\end{proof}

\begin{lemm}\label{L-fibre}
Let $r\ge 1$. Let $(L,\mathcal L)$ be a $d$-dimensional lamination $C^r$-immersed by $i$ into a Riemannian manifold $(M,g)$. Let $N$ be a $C^r$-lifting of $i$ into the Grassmannian of $k$-planes of $TM$. If for any $x\in L$, $N(x)$ is in direct sum with $Ti(T_x\mathcal L)$ then there exists a structure  of lamination $\mathcal F$ on the vector bundle $\pi\; :\; F\rightarrow L$ whose fiber at $x\in L$ is $N(x)$, such that:
\begin{itemize}\item[-] $\pi$ is a $C^r$-submersion,
\item[-] $I\; :\; (x,u)\in F\mapsto \exp_x(u)$ is a $C^r$-immersion on a neighborhood of the $0$-section of $F$,
\item[-] $(F,\mathcal F)$ is of dimension $d+k$.\end{itemize}
\end{lemm}  

\begin{proof}
For $x\in L$, let $\phi\in \mathcal L$ be a chart of a neighborhood $U$ of $x$. We may suppose that $\phi$ can be written in the form:
\[\phi\;: \; U\rightarrow \mathbb R^d\times T,\]
where $d$ is the dimension of the leaves of $\mathcal L$ and $T$ is a locally compact metric space. 

We may also suppose $U$ small enough to be sent by $i$ into a distinguish open subset of $M$ that we can identify to $\mathbb R^n$. Therefore, the restriction of the tangent space of $M$ to this distinguish open is identified to $\mathbb R^n\times \mathbb R^n$. Finally, in this identification, we suppose $F_x$ identified to $\{0\}\times \mathbb R^{k}\times \{0\}$. Let $p$ be the canonical projection of $\mathbb R^n\times \mathbb R^n$ onto $\{0\}\times (\mathbb R^{k}\times \{0\})$. For $U$ small enough, $p$ is a linear bijection from $F_y$ onto $\mathbb R^{k}$, for every $y\in U$.

We can now define the chart:
\[\phi_{U,x}\; :\; \pi^{-1}(U)\rightarrow \mathbb R^d\times \mathbb R^{k} \times T\]
\[(y,v)\mapsto (\phi_1(y),   p(v)           ,\phi_2(y))\]
where $\phi_1$ and $\phi_2$ are the first and the second coordinates of $\phi$.

As we already saw, we can identify the restriction of the Grassmannian of $TM$ at a neighborhood of $F_x$ to the product of $\mathbb R^n$ with the space $L(\mathbb R^{k}, \mathbb R^{n-k})$ of linear maps from $\mathbb R^{k}$ to $\mathbb R^{n-k}$. In such  identifications, for $U$ sufficiently small, the restriction of $F$ to $U$ is a $C^r$-morphism from $U$ into $\mathbb R^n\times L(\mathbb R^{k}, \mathbb R^{n-k})$, which is of the form $F(y)=(i(y),l_y)$.

For another such a chart $\phi_{U',x'}=(\phi_1',p', \phi_2')$, the changing coordinate:
 
\[\phi_{U',x'}\circ \phi_{U,x}^{-1}\; :\; \mathbb R^d\times \mathbb R^{k} \times T\rightarrow \mathbb R^d\times \mathbb R^{k} \times T'\]
\[(u,w,t)\mapsto (\phi_1'(y),   p'\circ l_y(w)           ,\phi_2'(y))\]
with $y:=\phi^{-1}(u,t)$, satisfies the requested properties in order that $(\phi_{U,x})_{x\in L}$ defines a $C^r$-lamination structure $\mathcal F$ on $F$.
\end{proof}

\bibliographystyle{alpha}
\nocite{*}
\bibliography{references}

\end{document}